\documentclass[11pt]{article}
\usepackage{amsmath,amsfonts,amssymb,amsthm,enumerate,graphicx}
\usepackage{subfigure}
\usepackage{amsmath}
\usepackage{amssymb}
\usepackage[german,english]{babel}
\usepackage{graphicx}
\usepackage{gastex}
\usepackage{longtable}
\usepackage{lscape}
\usepackage{multicol}
\usepackage{latexsym}
\usepackage{verbatim}
\usepackage{multirow}
\usepackage{authblk}

\usepackage{amssymb,latexsym,amsmath,epsfig,amsthm,mathrsfs}
\usepackage{amsfonts}
\usepackage{amscd}
\usepackage{dsfont}
\usepackage{bbm}
\usepackage{rotating}
\usepackage{newlfont}
\usepackage{enumitem}
\usepackage{lineno,hyperref}
\usepackage[english]{babel}
\usepackage{tabularx}
\usepackage{tikz}
\usepackage{tkz-graph}
\usepackage{longtable}
\usepackage{tabularx,ragged2e,booktabs,caption}
\tikzstyle{vertex}=[ draw, inner sep=0pt, minimum size=0pt]

\usepackage{epstopdf}
\usepackage{url}
\usepackage{hyperref}
\usepackage{float}
\usepackage{enumitem}
\usepackage{mathrsfs}
\usepackage[mathscr]{euscript}
\usepackage{mathtools}

\DeclarePairedDelimiter{\floor}{\lfloor}{\rfloor}
\newlist{subquestion}{enumerate}{1}
\date{}
\usepackage{soul}
\usepackage{xfrac}
\usepackage[normalem]{ulem}
\usepackage{braket}

\newtheorem{theorem}{{\bf Theorem}}[section]

\newtheorem{definition}[theorem]{{\bf Definition}}

\numberwithin{Subcase}{Case}

\numberwithin{Subsubcase}{Subcase}

\newtheorem{lem}{{\sc Lemma}}[section]

\newtheorem{prop}{{\sc Proposition}}[section]

\newcommand{\Echar}[1]{Euler characteristic #1}

\newcommand{\vx}{vertex}

\begin{document}
	\author[1] {Debashis Bhowmik}
	\author[2] {Dipendu Maity}
	\author[1] {Bhanu Pratap Yadav}
	\author[1] {Ashish Kumar Upadhyay}
	\affil[1]{Department of Mathematics, Indian Institute of Technology Patna, Patna 801\,106, India.
		\{debashis.pma15, bhanu.pma16, upadhyay\}@iitp.ac.in.}
	\affil[2]{Department of Sciences and Mathematics,
		Indian Institute of Information Technology Guwahati, Bongora, Assam-781\,015, India.
		dipendu@iiitg.ac.in/dipendumaity@gmail.com.}
	
	\title{New Classes of Quantum Codes Associated with Surface Maps}
	
	
	\date{\today}

	\maketitle
	
	\vspace{-10mm}
		\begin{abstract}
		If the cyclic sequences of {face types} {at} all vertices in a map are the same, then the map is said to be a semi-equivelar map. In particular, a semi-equivelar map is equivelar if the faces are the same type. Homological quantum codes represent a subclass of topological quantum codes. In this article, we introduce {thirteen} new classes of quantum codes. These codes are associated with the following: (i) equivelar maps of type $ [k^k]$, (ii) equivelar maps on the double torus along with the covering of the maps, and (iii) semi-equivelar maps on the surface of \Echar{-1}, along with {their} covering maps. The encoding rate of the class of codes associated with the maps in (i) is such that $ \frac{k}{n}\rightarrow 1 $ as $ n\rightarrow\infty $, and for the remaining classes of codes, the encoding rate is $ \frac{k}{n}\rightarrow \alpha $ as $ n\rightarrow \infty $ with $ \alpha< 1 $.
	\end{abstract}
\noindent {\small {\em MSC 2010\,: } 94Bxx
	
	}

	\section{Introduction}
\label{introduction}

Quantum error{-}correction codes (QECs) have been developed to protect quantum information from decoherence and quantum noise. In 1995, Shor \cite{shor1995} {became the} first {to} introduce QECs{. Additionally,} {in 1998}, Calderbank et al. \cite{calderbank1998} proposed a systematic way to create QECs from classical error-correcting codes. {In 1997, Gottesman {\cite{Gottesman1997}} introduced \textit{stabilizer code}, a tool to describe quantum codes like linear codes in a classical setting.} Topological quantum computation is an analog type of quantum computation with the benefit of being inherently fault-tolerant because of the topological properties of the physical system. {The topological quantum code (TQC) constitutes a class of QECs that are constructed based on linear code structure. Such codes are related to tessellations of a surface.} {This topological quantum code was implemented by Kitaev \mbox{\cite{kitaev2003}} in 2003.}

In this article, a surface will mean a connected{,} compact 2-manifold without boundary. A cellular embedding of a simple finite graph on a surface is called a map. Semi-equivelar maps are generalizations of Archimedean solids to the surfaces other than the 2-sphere. If $G$ denotes an embedded graph on a surface {M}, then the vertices and edges of the graph $G$ are called vertices and edges of the map, and closures of the connected components of {$M \setminus G$} are called faces. Let $ \mathcal{K} $ {be} a map and $ V(\mathcal{K}) $ denote the set of vertices of $ \mathcal{K} $. For $u\in V(\mathcal{K}) $, the faces containing  $u $ form a cycle (called the {\em face-cycle} at  $u $)  $C_u $ in the dual $ \mathcal{K}^* $ of $\mathcal{K}$. By combining neighboring polygons with the same number of vertices, $C_u $ can be decomposed in the form $ B_1-B_2-\dots -B_k-A_{1,1} $, where $ B_i=A_{i,1}-\dots -A_{i,n_i} $ is a path consisting of $ a_i $-gons $ A_{i,1},\dots,A_{i,n_i} $ for $ 1\leq i\leq k $, $a_r\neq a_{r+1} $ for  $1\leq r\leq k-1 $ and  $a_k\neq a_1 $. In this case, we say that $ u $ is of type $[a_1^{n_1}, \dots, a_k^{n_k}] $. A map $ \mathcal{K} $ is said to be a \textit{semi-equivelar map} {(SEM)} if for any $ u, v \in V(\mathcal{K}) $, $ C_u $ and $ C_v $ are of the same type. We can say that the {SEM} is a map of type $ [a_1^{n_1}, \dots, a_l^{n_l}] $. In particular, if the type of a map is $ [p^q]$, then it is called an \textit{equivelar map}. We suggest \cite{dm2018} and \cite{bmu2020} to the reader for further information on this topic.

We define the $ d $-th cover map of $ \mathcal{K} $ as follows: let $ C $ be a non-separable cycle of length $ k $ in $ \mathcal{K} $. By cutting $ \mathcal{K} $ along $ C $, we then have a map $ \mathcal{K}_* $ with two boundary cycles $ A(p_1,p_2,\dots, p_k) $ and $ B(q_1,q_2,\dots, q_k) $, where identifying $ A $ with $ B $ by the map $ p_i\longmapsto q_i $ ($ i=1,2,\dots, k $) {produces} the cycle $ C $, and hence $ \mathcal{K} $. We consider two copies of $ \mathcal{K}_* $, {and} denote {them} as $ \mathcal{M}_1 $ and $ \mathcal{M}_2 $. Thus $ \mathcal{M}_1 $ {has} boundary cycles{, namely,} $ A_1(p_1^1,p_2^1,\dots, p_k^1) $, $ B_1(q_1^1,q_2^1,\dots, q_k^1) ${,} and {$\mathcal{M}_2$ has boundary cycles}{, namely,} $ A_2(p_1^2,p_2^2,\dots, p_k^2) $, $ B_2(q_1^2,q_2^2,\dots, q_k^2) $. By identifying $ A_1(p_1^1,p_2^1,\dots, p_k^1) $ with $ B_2(q_1^2,q_2^2,\dots, q_k^2) $ by the map $ p_i^1\longmapsto q_i^2 $, and $ B_1(q_1^1,q_2^1,\dots, q_k^1) $ with $ A_2(p_1^2,p_2^2,\dots, p_k^2) $ by the map $ q_i^1\longmapsto p_i^2 $, we obtain a map on {the surface of \Echar{} $ 2\chi(\mathcal{K}) $ with $ 2 \times |V(\mathcal{K})|$ number of vertices{, where $\chi(\mathcal{K})$ is the Euler characteristic of $\mathcal{K}$, and $|V(\mathcal{K})|$ denotes the cardinality of the set $|V(\mathcal{K})|$}.} Again, we consider $ d $ copies of $ \mathcal{K}_* $ as $ \mathcal{M}_j $ with cycles $ A_j,B_j $ for $ j=1,2,\dots d $ ($ d\geq1 $). As before, identifying $ B_i $ with $ A_{i+1} $ for $1 \le i \le d-1${,} and $ B_d $ with $ A_1 $, we obtain a map of the same type{.} {This map is called} the {\em $ d $-th cover} of $ \mathcal{K} $. We denote this map as $ \mathcal{K}^d $. Clearly, $d \times |V(\mathcal{K})| $ is the number of vertices of $ \mathcal{K}^d $ with \Echar{} $ d\chi(\mathcal{K})$. 	

{We denote the Galois field $ GF(q) $ by $ \mathbb{F}_q $, where $ q $ is a prime integer. Hence, $ \mathbb{F}_q^n $ is a vector space of dimension $ n $, where $ n $ is a positive integer. Letting $ x,y\in\mathbb{F}_q^n $ and $ x=(x_1,x_2,\dots, x_n) $, $ y=(y_1,y_2,\dots, y_n) $. Then, the product of $ x $ and $ y $ is defined as}
\begin{equation}
{x\cdot y=\sum_{i=1}^{n}x_iy_i} \nonumber
\end{equation}

{A \textit{linear code} \mbox{\cite{hill}} over $\mathbb{F}_q$ is a subspace of $\mathbb{F}_q^n$. That is, a subset $ \mathcal{C} $ of $ \mathbb{F}_q^n $ is said to be a linear code if and only if}
\begin{enumerate}
	\item {$ x+y\in \mathcal{C} $ for all $ x,y\in \mathcal{C} $,}
	\item {$ ax\in \mathcal{C} $ for all $ x\in \mathcal{C} $ and $ a\in \mathbb{F}_q $.}
\end{enumerate}
{If $ q=2 $, then the code $ \mathcal{C} $ is called \textit{binary linear code}. The elements of $ \mathcal{C} $ are called \textit{codewords}. If $ \mathcal{C} $ is a $ k $-dimensional ($ <n $) subspace, then $ \mathcal{C} $ is called a $ [[n,k]] $ code. The \textit{weight} of a codeword $ x\in\mathcal{C} $ is the number of non-zero elements, and is denoted as $ w(x) $. The distance between two codewords $ x,y\in \mathcal{C} $ (also called the Hamming distance) is the number of positions $ x $ and $ y $ which differ. It is denoted by $ d_H(x,y) $. Using this distance, the code distance is defined. The \textit{distance} of a code $ \mathcal{C} $ is denoted as $ d_{min} $ (or $ d(\mathcal{C}) $), and defined by}
\begin{equation}
{d_{min}=min\{d_H(x,y):x,y\in \mathcal{C}, x\neq y\}} \nonumber
\end{equation}
{In this case, code $\mathcal{C}$ is referred to as $ [[n,k,d_{min}]] $ code. A matrix $ G\in M_{n\times k}(\mathbb{F}_q) $ exists, where $ G $ is said to be the \textit{generator matrix} of $ \mathcal{C} $ if $ \mathcal{C}=G(\mathbb{F}_q^k) $, where $M_{n\times k}(\mathbb{F}_q)$ denotes the set of $n\times k$ matrices over $\mathbb{F}_q$. A matrix $ H\in M_{n\times k}(\mathbb{F}_q) $ is called the \textit{parity-check matrix} of the code $ \mathcal{C} $ if $ \mathcal{C}=ker(H) $.}

{Since $ \sfrac{\mathbb{F}_q^n}{ker(H)}\cong Im(H) $, then $~ dim(Im(H))~=$ $n-k $, i.e., $ H $ does not need to be full rank. Moreover, $ HG=0 $, as columns of $ G $ are elements of $ \mathcal{C} $. Therefore, the generating matrix $ G $ for a code $ \mathcal{C} $ can be constructed by considering a basis of $ \mathcal{C} $, and placing it as column of $ G $. To make a full rank parity-check matrix of $ \mathcal{C} $, we take into account the orthogonal complement}
\begin{equation}
{\mathcal{C}^{\perp}=\{y\in \mathbb{F}_q^n: y\cdot x=0 \text{ for all }x\in \mathcal{C}\}. \nonumber}
\end{equation}
{$\mathcal{C}^{\perp}$ is called the \textit{dual code} of $ \mathcal{C} $. It is not necessary that $ \mathcal{C}\cap\mathcal{C}^\perp=\{0\} $. It is easy to determine that $ \mathcal{C}^\perp $ is a $ [[n,n-k]] $ code with $ H^T $ as generator matrix, and $ G^T $ as parity check matrix.}

{\textit{Bit} \mbox{\cite{MichaelA.Nielsen2004}} is the fundamental concept of classification computation.
	Quantum computation is based on \textit{qubits} (or \textit{quantum bits}), and is denoted by the symbol `$ \ket{} $'. Because a bit has two states, namely 0 or 1, a qubit also has these states. Two probable states for each qubit are $ \ket{0} $ and $ \ket{1} $. This can also result from a linear combination of states, such as }
\begin{equation}
{\ket{\sigma}=\alpha\ket{0}+\beta\ket{1} \nonumber}
\end{equation}
{for some $ \alpha,\beta\in\mathbb{C} $. States of the qubits can also be considered as vectors of $ \mathbb{C}^2 $ with basis (orthogonal) $ \ket{0} $ and $ \ket{1} $. The state of $ n $ qubits is an element of $ 2^n $-dimensional vector space $ \mathcal{H}_n=(\mathbb{C}^2)^{\otimes n} $. Thus, for $ \sigma \in \mathcal{H}_n $,}
\begin{equation}
{\ket{\sigma}=\sum_{j\in \mathbb{F}_2^n}a_j\ket{j} \nonumber}
\end{equation}
{where $ \ket{j}=\ket{j_1}\otimes\ket{j_2}\otimes\dots\otimes\ket{j_n} $, and $ \sum|a_i^2|=1 $. In particular, $ \ket{\sigma_{1}} $ at time $ t_1 $ and $ \ket{\sigma_{2}} $ at time $ t_2 $ are associated by the Unitary transformation $ U=U(t_1,t_2)\in U(\mathcal{H}_n) $ such that $ \ket{\sigma_{1}}=U\ket{\sigma_{2}} $, where $ U(\mathcal{H}_n) $ is the group of unitary operators on the space of $ n $ qubits.}

{The Pauli group on $ n $ qubits is the set $P_n=\{c\bigotimes_{i=1}^{n}A_i: A_i\in$ $\{I,X,Y,Z \}, c\in \{\pm 1, \pm i\}\}$, where $ I,X,Y,X $ represents the Pauli matrix and is given as follows:}
\begin{equation*}
X =
	\begin{pmatrix}
	0 & 1 \\
	1 & 0 \\
	\end{pmatrix},
	\hspace{.2cm}
	Y=
	\begin{pmatrix}
	0 & -i \\
	i & 0
	\end{pmatrix},
	\hspace{.2cm}
Z=
	\begin{pmatrix}
	1 & 0 \\
	0 & -1
	\end{pmatrix},
	\hspace{.2cm}
	I=
	\begin{pmatrix}
	1 & 0 \\
	0 & 1
	\end{pmatrix}.
\end{equation*}
{The stabilizer group $ \mathcal{S} $ is the subgroup of $ P_n $ with $ -I\notin \mathcal{S} $. A stabilizer code $ \mathcal{C} $ is thus the simultaneous eigenspace associated with $ \mathcal{S} $, i.e., $ \mathcal{C}=\{\ket{\psi}:P\ket{\psi}=\ket{\psi} \forall P\in \mathcal{S}\} $. Moreover, the \textit{normalizer} of $ \mathcal{S} $ is $N_{P_n}(\mathcal{S})=\{E\in P_n:$ $EgE^\dagger$ $\in \mathcal{S}$ for all $g\in\mathcal{S}\}$, and the \textit{centralizer} of $ \mathcal{S} $ is $C_{P_n}(\mathcal{S})=\{E\in P_n: EgE^\dagger=g\text{ for all } g\in\mathcal{S}\}$, where $E^\dagger$ is the Hermitian adjoint of the operator $ E $.}

{The distance $ d_{min} $ of the stabilizer code $ \mathcal{C} $ is the minimum weight of Pauli operators in $ C_{P_n}(\mathcal{S})\setminus \mathcal{S} $. It can correct errors up to $ \lfloor \frac{d_{min}-1}{2} \rfloor $ qubits (see {\cite{leslie2014}}), where $ \floor{*} $ is the usual floor function. If $ \mathcal{C} $ is a $ 2^k $ dimensional subspace of $ \mathcal{H}_n $ with distance $ d_{min} $, then the code $ \mathcal{C} $ is referred to as the $ [[n,k,d_{min}]] $ quantum code. Here, if the encoding rate $\frac{k}{n}$ of the $[[n,k,d_{min}]]$ quantum code $\mathcal{C}$ tends to $1$, then $\mathcal{C}$ is considered as a good code. In particular, parameters of the Kitaev's toric code are $[[2m^2,2,m]]$: therefore, its encoding rate is $0$.}

\textit{Homological Quantum Code} {(HQC)} is a subclass of TQCs. It was introduced by Bombin and Martin-Delgado \cite{bombin2006}. They also presented some {HQC} on the surface of an arbitrary genus. This article is an attempt to present several new classes of {HQC} associated with equivelar {map} and {SEM}. These codes are constructed by two parity-check matrices $ H_X $ and $ H_Z $ such that $ H_XH_Z^T=0 $, where entries of $ H_X $, $ H_Z $ are from $ \mathbb{Z}_2 $.

In this paper, we build {thirteen} new classes of {HQCs}.
{The quantum codes associated with equivelar maps are} as follows:
\begin{enumerate}
	\item $ [[(2m_1-1)(3^{m_1-1}+2m_2-1),2+(2m_1-5)(3^{m_1-1}+2m_2-1), 4 ]],  m_1\geq3, m_2\geq0, $
	\item  $ [[m_1(3^{m_1}+2m_2-1), 2+(m_1-2)(3^{m_1}+2m_2-1), 4]], m_1\geq3, m_2\geq0; $
\end{enumerate}
the encoding rates of these two codes are such that $ \frac{k}{n}\rightarrow 1 $ as $ n\rightarrow \infty $. We also present{ed} {eleven} more classes of quantum codes{,} {namely}, $ [[42d,2(1+d),3]] $, $ [[40d,2+d,4]] $, {$[[84d,2+d,4]]$, $[[63d,2+d,4]]$, $[[126d,2+d,4]]$, $[[60d,2+d,4]]$,} $ [[72d,2+d,4]] $, $ [[48d,2+d,4]] $, $ [[36d,2+d,3]] $, $ [[30d,2+d,4]] $, and $ [[36d,2+d,3]] $ with $ d\geq 1 $. In this list, the first one is associated with the $ d $-th cover map of the equivelar maps on a double torus, and the remaining codes are associated with the $ d $-th cover map of the {SEMs} on the surface with Euler characteristic $-1$. The encoding rates of these codes are less than 1.

This paper is organized in the following manner. In Section \ref{tqc}, we provide an outline of TQCs, and in Section \ref{homology quantum code}, we present an overview of {HQC} for a general complex and surface maps. In Section \ref{kk equivelar code}, we present the {class of HQCs which are derived from the} equivelar maps of type $ [k^k]$. In Section \ref{37 quantum code}, we provide an example to calculate codes associated with equivelar maps and {SEMs}. We produce a list of codes associated with equivelar maps on a double torus and {SEMs} on the surface of \Echar{-1} and the covering maps. Finally, in Section \ref{table of quantum codes}, we compare our codes with existing quantum codes which are available in the literature, and we conclude this paper in Section \ref{conclution}.
\section{Topological Quantum code}\label{tqc}

%
%

{A regular tessellation of a Euclidean or hyperbolic plane is a partition by regular polygons, all with the same number of edges, for which the intersection of two polygons is either empty, a vertex, or an edge. We denote a regular tessellation by $[p^q]$, where $q$ number of regular polygons with $p$ edges meet at each vertex, and the dual map by $[q^ p]$. More details can be found in \mbox{\cite{Gruenbaum1987}}.}

{TQC is a subclass of stabilizer codes.} In general, TQC is defined as follows.
\begin{definition}[\cite{silva2010}]
	{ Let $ M $ be a map (or $[p^q]$  tessellation). Let $ V $, $ E $ and $ F $ denote the vertex set, edge set and face set, respectively, of $M$. Given a vertex $ v\in V $ and a face $ f\in F $, we define an operator $ A_v $ as the tensor product of $ X $ corresponding to the edges incident to $ v $, and $ B_f $ as the tensor product of $ Z $ corresponding to the edges on the border of $ f $. Hence, the topological quantum code $[[n, k, d_{min}]]$ is defined where the code length $ n = |E| $, the stabilizer $ \mathcal{S}=\{A_v|v\in V\}\cup\{B_f|f\in F\} $, the number of encoded qubits $ k = 2-\chi(M) $, and the code distance $ d_{min} = min\{\delta, \delta^*\} $, where $ \delta $ denotes the code distance in $M$,
		and $ \delta^* $ denotes the code distance in the dual map $M^*$.}
\end{definition}
Therefore, for a tessellation $ [p^q] $ of a closed, compact surface $ M $, \[ A_v=\bigotimes_{e\in E_v}X^e,~~~~~B_f=\bigotimes_{e\in E_f} Z^e \]
where $ E_v $ denotes the set of edges incident to the vertex $ v $, and $ E_f $ denotes the set of edges adjacent to the faces $ f $. {Thus, $ \mathcal{C}=\{\ket{\psi}: A_v\ket{\psi}=\ket{\psi}\forall v\in V\}\cup \{ \ket{\psi}: B_f\ket{\psi}=\ket{\psi}\forall f $ $\in F \} $.
	The minimum distance of this code, denoted by $ d_{min} $, is the shortest non-contractible cycle length, i.e., the number of edges of the shortest non-contractible cycle in the tessellations $ [p^q] $ and $ [q^p] $.} 
\section{Homological quantum codes associated with maps} \label{homology quantum code}
\begin{prop}[\cite{avaz2018,silva2009}]\label{quantum code define condition}
	Let $ \mathcal{C}_X $ and $ \mathcal{C}_Z $ be two classical binary linear codes of length $ n $ with parity-check matrices $ H_X $ and $ H_Z $, respectively. If $ H_XH_Z^T=0 $, then the stabilizer code with binary check matrix
	\begin{eqnarray}
	A=\begin{bmatrix}
	H_X & 0 \\
	0 & H_Z
	\end{bmatrix} \nonumber
	\end{eqnarray}
	is a $ [[n,k,d_{min}]] $ quantum code{,} where $ n $ is the codeword length, $ k=n-\dim\mathcal{C}_X^\perp-\dim{\mathcal{C}_Z^\perp} $, $ d_{min}= min\{wt(x):x\in (\mathcal{C}_Z\setminus \mathcal{C}_X^\perp)\cup (\mathcal{C}_X\setminus \mathcal{C}_Z^\perp) \} $.
\end{prop}
Let $ (C_{\bullet},\partial_{\bullet}) $ be a chain complex associated with a complex $ K $
\begin{eqnarray}\label{chain complex}
\cdots\xrightarrow{} C_{i+1}\xrightarrow{\partial_{i+1}} C_i\xrightarrow{\partial_{i}} C_{i-1}\xrightarrow{}\cdots.
\end{eqnarray}
Then, $ H_i(K)=\frac{Ker(\partial_i)}{Im(\partial_{i+1})} $ is called the $ i $-th Homology of $K$, as in \cite{hatcher}. In particular, if the chain complex is over $ \mathbb{Z}_2 $, then $ H_i(K) =  H_i(K,\mathbb{Z}_2) ${,} and if it is over $ \mathbb{Z} $, then $ H_i(K) = H_i(K,\mathbb{Z}) $.

Let $ H_X=[\partial_i]_{p\times q} $ and $ H_Z=[\partial_{i+1}]^T_{r\times q} $, where $ p=\dim C_{i-1} $, $ q= \dim C_i$, $ r=\dim C_{i+1} $ and $ [\partial_j] $ is a matrix corresponding to linear map $ [\partial_j] $ with respect to a basis of $ C_j $ and $ C_{j-1} $. Clearly, $ H_XH_Z^T=[\partial_i\circ \partial_{i+1}]=0  $. Thus, by Proposition \ref{quantum code define condition}, {we can construct a $ [[n,k,d_{min}]] $ quantum code}, where $ n=\dim C_i $, $ k=n-\dim\mathcal{C}_X^\perp-\dim{\mathcal{C}_Z^\perp}=\dim (H_i) $, and $ d_{min}= min\{wt(x):x\in(\mathcal{C}_X\setminus \mathcal{C}_Z^\perp)\cup (\mathcal{C}_Z\setminus \mathcal{C}_X^\perp) \} $. Note that $ H_i=\frac{Ker(\partial_i)}{Im(\partial_{i+1})}=\frac{Ker(H_X)}{Im(H_Z^T)}= \mathcal{C}_X\setminus \mathcal{C}_Z^\perp $, and by dualizing (\ref{chain complex}), we can obtain the $ i $-th Cohomology $ H^i=\frac{Ker(H_Z)}{Im(H_X^T)}=\mathcal{C}_Z\setminus \mathcal{C}_X^\perp $. For more information on this, see \cite{avaz2018, silva2009, tillich2009}.

Let $ M $ be a surface of genus $ g $, and $ \mathcal{K} $ be a map embedded on $ M $ with $ f_0 $ {($=V(\mathcal{K})$)}, $ f_1 $ and $ f_2 $ as vertex, edge and face set, respectively. Let $ \mathscr{V} $, $ \mathscr{E} $, and $ \mathscr{F} $ be $ \mathbb{Z}_2 $-vector spaces with basis $ f_0 $, $ f_1 $ and $ f_2 $, respectively. We can now define a chain complex
\begin{eqnarray}
\{0\} \rightarrow \mathscr{F}\xrightarrow{\partial_2}\mathscr{E}\xrightarrow{\partial_1}\mathscr{V}\rightarrow \{0\} \nonumber
\end{eqnarray}
with $ \partial_1\circ\partial_2=0 $. Hence, $ H_XH_Z^T=[\partial_1\circ\partial_2]=0 $, where $ H_X=[\partial_1] $, i.e., the vertex-edge incident matrix of $ \mathcal{K} $, and $ H_Z=[\partial_2]^T $, i.e., the face-edge incident matrix of $ \mathcal{K} $. Thus, the $ [[n,k,d_{min}]] $ {quantum code is obtained} with $ n=|f_1| $, $ k=\dim (H_1(\mathcal{K},\mathbb{Z}_2)) $, and $ d_{min}= min\{wt(x):x\in (\mathcal{C}_X\setminus \mathcal{C}_Z^\perp)\cup (\mathcal{C}_Z\setminus \mathcal{C}_X^\perp) \} $, i.e., the length of the shortest non-contractible cycle of $ \mathcal{K} $ or $ \mathcal{K}^* $, where $ \mathcal{K}^* $ is the dual map of $ \mathcal{K} $. Now, if $ H_2(\mathcal{K},\mathbb{Z}) \cong \mathbb{Z} $, then $ H_1(\mathcal{K}, \mathbb{Z}_2)\cong H_1(M,\mathbb{Z}_2)\cong(\mathbb{Z}_2)^{2g} $, and if $ H_2(\mathcal{K},\mathbb{Z})\cong \{0\} $, then $ H_1(\mathcal{K}, \mathbb{Z}_2)\cong H_1(M,\mathbb{Z}_2) \cong (\mathbb{Z}_2)^{g} $. Therefore, for $ H_2(\mathcal{K},\mathbb{Z}) \cong \mathbb{Z} $, $ k=2g $, and for $ H_2(\mathcal{K},\mathbb{Z})\cong \{0\} $, $ k=g $. However, $ H_2(\mathcal{K},\mathbb{Z}) \cong\mathbb{Z} $ implies that $ 2g = 2-\chi(\mathcal{K}) $, and $ H_2(\mathcal{K},\mathbb{Z}) \cong \{0\} $ implies that $ g=2-\chi(\mathcal{K}) $. Therefore, for any $ H_2(\mathcal{K},\mathbb{Z}) $, $ k=2-\chi(\mathcal{K}) $.

\section{Codes associated with equivelar maps of type {$ \lowercase{[k^k]} $}} \label{kk equivelar code}
We know the existence of equivelar maps of type $ [k^k] $ {from \mbox{\cite{dutta05}}}. In this section, we construct a class of homological quantum codes $ [[n,k,d_{min}]] $ associated with equivelar maps of type $ [k^k] $.
\begin{prop}[\cite{dutta05}]\label{exist_EM}
	For each $ m_1\geq 3 $ and $ m_2\geq0 $, there exists an $ 2 \times (3^{m_1-1}+2m_2-1) $-\vx{} self-dual equivelar map of type $ [{(2m_1-1)}^{2m_1-1}] $ and an $ (3^{m_1}+2m_2-1) $-\vx{} self-dual equivelar map of type $ [{(2m_1)}^{2m_1}] $.
\end{prop}
We establish a result in Lemma \ref{dmin for 2m1-1} to produce quantum codes associated with equivelar maps of types $ [{(2m_1-1)}^{2m_1-1}] $ and $ [{(2m_1)}^{2m_1}] $.
\begin{lem}\label{dmin for 2m1-1}
	If equivelar maps are of type $ [{(2m_1-1)}^{2m_1-1}] $ or $ [{(2m_1)}^{2m_1}] $, then $ d_{min}=4$.
\end{lem}
\begin{proof}
	From \cite{dutta05}, for $ m_1\geq3 $ and $ m_2\geq 0 $, equivelar maps of types $ [{(2m_1-1)}^{2m_1-1}] $ and $ [{(2m_1)}^{2m_1}] $ are as follows:
		\begin{eqnarray}
	\mathcal{K}_{2m_1-1,2 \times (3^{m_1-1}+2m_2-1)} & = & \{F_{j,2m_1-1}: 1\leq j \leq 2 \times (3^{m_1-1}+2m_2-1)\} \nonumber \\
	\mathcal{K}_{2m_1,3^{m_1}+2m_2-1} & = & \{F_{j,2m_1}: 1\leq j \leq 3^{m_1}+2m_2-1 \} \nonumber
	\end{eqnarray}
	where $~ F_{j,2m_1-1} ~$ and $ ~F_{j,2m_1} ~$ are faces of $ \mathcal{K}_{2m_1-1,2 \times (3^{m_1-1}+2m_2-1)} $ and $ \mathcal{K}_{2m_1,3^{m_1}+2m_2-1} $, respectively{,} and are given by
	\begin{eqnarray}
	F_{j,2m_1-1} & = & (j+a_1, j+a_2, \dots , j+a_{2m_1-3}, j+a_{2m_1-2}+m_2, j+a_{2m_1-1}+2m_2) \nonumber \\
	F_{j,2m_1} & = & (j+a_1, j+a_2, \dots , j+a_{2m_1-2}, j+a_{2m_1-1}, j+a_{2m_1}+m_2) \nonumber
	\end{eqnarray}
	with $ a_{2n-1}=3^{n-1}-1 $, $ a_{2n}=2\times 3^{n-1}-1 $ for $ n\geq 1 $. Therefore, $ F_{j,2m_1-1} $ and $ F_{j,2m_1} $ are $ (2m_1-1) $-cycles and $ (2m_1) $-cycles, respectively, with vertices from $ \mathbb{Z}_{2 \times (3^{m_1-1}+2m_2-1)} $ and $ \mathbb{Z}_{3^{m_1}+2m_2-1} $, respectively. Thus, every vertex in $ \mathcal{K}_{2m_1-1,2 \times (3^{m_1-1}+2m_2-1)} $ contains $ 2m_1-1 $ faces, and every vertex in $ \mathcal{K}_{2m_1,3^{m_1}+2m_2-1} $ contains $ 2m_1 $ faces. Now we have
			\begin{eqnarray}
			F_{1,2m_1-1} &=& (1,2,3,\dots , 3^{m_1-2}, 2\times 3^{m_1-2}+m_2, 3^{m_1-1}+2m_2) \label{eq:F1(2m1-1)} \\
			F_{2,2m_1-1} &=& (2,3,4, \dots , 2\times 3^{m_1-1}+m_2+1, 3^{m_1-1}+2m_2+1) \label{eq:F2(2m1-1)} \\
			F_{3^{m_1-1}+2m_2-1,2m_1-1} &=& (3^{m_1-1}+2m_2-1, 3^{m_1-1}+2m_2, 3^{m_1-1}+2m_2+1,  \nonumber\\
			&&3^{m_1-1}+2m_2+4, \dots , 2\times(3^{m_1-1}+2m_2-1)) \label{eq:F3(2m1-1)}\\
			F_{3^{m_1-1}+2m_2,2m_1-1} &=& (3^{m_1-1}+2m_2, 3^{m_1-1}+2m_2+1, 3^{m_1-1}+2m_2+2, \dots , \nonumber \\
			&&2\times(3^{m_1-1}+2m_2)-1) \label{eq:F4(2m1-1)}\\
			F_{3^{m_1-1}+2m_2+1,2m_1-1} &=& (3^{m_1-1}+2m_2+1, 3^{m_1-1}+2m_2+2, 3^{m_1-1}+2m_2+3, \dots , 2) \label{eq:F5(2m1-1)}\\
			F_{1,2m_1} &=& (1,2,3, \dots , 2\times 3^{m_1-2}, 3^{m_1-2}, 2\times 3^{m_1-1}+m_2) \label{eq:F1(2m1)} \\
			F_{2,2m_1} &=& (2,3,4, \dots , 3^{m_1-1}+1, 2\times 3^{m_1-1}+m_2+1) \label{eq:F2(2m1)} \\
			F_{2\times 3^{m_1-1}+m_2-1,2m_1} &=& (2\times 3^{m_1-1}+m_2-1, 2\times 3^{m_1-1}+m_2, 2\times 3^{m_1-1}+m_2+1, \nonumber \\
			&& 2\times 3^{m_1-1}+m_2+6, \dots, 4\times 3^{m_1-1}+2m_2-2) \label{eq:F3(2m1)} \\
			F_{2\times 3^{m_1-1}+m_2,2m_1} &=& (2\times 3^{m_1-1}+m_2, 2\times 3^{m_1-1}+m_2+1, 2\times 3^{m_1-1}+m_2+2, \dots , \nonumber \\
			&& 4\times 3^{m_1-1}+2m_2-1) \label{eq:F4(2m1)} \\
			F_{2\times 3^{m_1-1}+m_2+1,2m_1} &=& (2\times 3^{m_1-1}+m_2+1, 2\times 3^{m_1-1}+m_2+2, 2\times 3^{m_1-1}+m_2+3, \nonumber \\
			&& \dots,3^{m_1-1}+1)  \label{eq:F5(2m1)}	
			\end{eqnarray}
	From Equations (\ref{eq:F1(2m1-1)}), (\ref{eq:F2(2m1-1)}) and (\ref{eq:F3(2m1-1)}), it is clear that $ C(3^{m_1-1}+2m_2, 1, 2, 3^{m_1-1}+2m_2+1) $ is a cycle in the equivelar map of type $ [{(2m_1-1)}^{2m_1-1}] ${,} and from Equations (\ref{eq:F1(2m1)}), (\ref{eq:F2(2m1)}), and (\ref{eq:F3(2m1)}), it is clear that $ C(2\times 3^{m_1-1}+m_2, 1, 2, 2\times 3^{m_1-1}+m_2+1 ) $ is a cycle in the equivelar map of type $ [{(2m_1)}^{2m_1}] $. Clearly, faces in equivelar maps of type $ [{(2m_1-1)}^{2m_1-1}] $ are $ (2m_1-1) $-gons, and faces in equivelar maps of type $ [{(2m_1)}^{2m_1}] $ are $ 2m_1 $-gons. Furthermore, $ (2m_1-1)\geq5 $ and $ 2m_1\geq 6 $ for $ m_1\geq3 $. Therefore, the cycles $ C(3^{m_1-1}+2m_2, 1, 2, 3^{m_1-1}+2m_2+1) $ and $ C(2\times 3^{m_1-1}+m_2, 1, 2, 2\times 3^{m_1-1}+m_2+1 ) $ are non-contractible in the respective equivelar maps, and hence $ d_{min}\leq 4 $ in both classes.
	
	{We know that if} $ C(c_1,c_2,c_3) $ is a cycle of length three in a map, then $ c_3 $ is adjacent to the vertex $ c_1 $. {Using this property, we} show that equivelar maps of {types} $ [{(2m_1-1)}^{2m_1-1}] $ and $ [{(2m_1)}^{2m_1}] $ {do} not contain a non-contractible cycle of length three. Here, it is sufficient to show that there does not exist any cycle of length three which contains the vertex $ 3^{m_1-1}+2m_2+1 $ ({resp. }$2\times 3^{m_1-1}+m_2+1$) in an equivelar {map} of type $ [{(2m_1-1)}^{2m_1-1}] $ ({resp. }$ [{(2m_1)}^{2m_1}] $). From the above faces (Eq. (\ref{eq:F1(2m1-1)}$ - $\ref{eq:F5(2m1)})), it is clear that the vertex $ 3^{m_1-1}+2m_2+1 $ is adjacent to the vertices $ 2 $, $ 3^{m_1-1}+2m_2+2 $, $ 3^{m_1-1}+2m_2 $, $ 3^{m_1-1}+2m_2+4 , \dots $, $ 2\times 3^{m_1-1}+m_2+1 $, and that the vertex $ 2\times 3^{m_1-1}+m_2+1 $ is adjacent to $ 3^{m_1-1}+1 $, $ 2\times 3^{m_1-1}+m_2+2 $, $ 2\times 3^{m_1-1}+m_2 $, $ 2\times 3^{m_1-1}+m_2+6,\dots $, $ 2 $. Therefore, $ 3^{m_1-1}+2m_2+1 $ is not adjacent to $ 1 $ or $ 3 $ in equivelar maps of type $ [{(2m_1-1)}^{2m_1-1}] $, and similarly the vertex $ 2\times 3^{m_1-1}+m_2+1 $ is not adjacent to $ 1 $ or $ 3 $ in an equivelar {map} of type $ [{(2m_1)}^{2m_1}] $. Therefore, equivelar maps of {types} $ [{(2m_1-1)}^{2m_1-1}] $ and $ [{(2m_1)}^{2m_1}] $ {do} not contain a non-contractible cycle of length three. Thus, $ d_{min}=4 $ in both the classes of equivelar maps of {types} $ [{(2m_1-1)}^{2m_1-1}] $ and $ [{(2m_1)}^{2m_1}] $. This completes the proof.
\end{proof}

\subsection{Codes associated with equivelar maps of type {$ \lowercase{[{(2m_1-1)}^{2m_1-1}]} $} }\label{code 2m1-1}
Let $ \mathcal{K} $ be an equivelar map of type $ [{(2m_1-1)}^{2m_1-1}] $. Then, by Proposition \ref{exist_EM}, $ |f_0|=2 \times (3^{m_1-1}+2m_2-1)$. Thus, $ |f_1|=(2m_1-1)(3^{m_1-1}+2m_2-1) $ and $ |f_2|=2 \times (3^{m_1-1}+2m_2-1) $. By \Echar{} equation, $ \chi(\mathcal{K})=|f_0|-|f_1|+|f_2|=(5-2m_1)(3^{m_1-1}+2m_2-1) $. Therefore, $ k=2-\chi=2+(2m_1-5)(3^{m_1-1}+2m_2-1) $, and by Lemma \ref{dmin for 2m1-1}, $ d_{min}=4 $. Thus, $ [[n,k,d_{min}]] =$ $ [[(2m_1-1)(3^{m_1-1}+2m_2-1), 2+(2m_1-5)(3^{m_1-1}+2m_2-1),4]] $. The encoding rate of the code $ \frac{k}{n}=\frac{2+(2m_1-5)(3^{m_1-1}+2m_2-1)}{(2m_1-1)(3^{m_1-1}+2m_2-1)} =$ $ \frac{2}{(2m_1-1)(3^{m_1-1}+2m_2-1)}+\frac{2-\frac{5}{m_1}}{2-\frac{1}{m_1}} $ $ \rightarrow 1 $ as $ m_1,m_2\rightarrow \infty $.
\subsection{Codes associated with equivelar maps of type {$ \lowercase{[{(2m_1)}^{2m_1}]} $}}\label{code 2m1}
Let $ \mathcal{K} $ be an equivelar map of type $ [{(2m_1)}^{2m_1}] $. Then, by Proposition \ref{exist_EM}, $ |f_0|=(3^{m_1}+2m_2-1)$. Thus, $ |f_1|=m_1(3^{m_1-1}+2m_2-1) $ and $ |f_2|=(3^{m_1}+2m_2-1) $. By \Echar{} equation, $ \chi(\mathcal{K})=|f_0|-|f_1|+|f_2|=(2-m_1)(3^{m_1}+2m_2-1) $.  Therefore, $ k=2+(m_1-2)(3^{m_1}+2m_2-1) $, and by Lemma \ref{dmin for 2m1-1}, $ d_{min}=4 $. Thus, $ [[n,k,d_{min}]]= $ $ [[m_1(3^{m_1}+2m_2-1),2+(m_1-2)(3^{m_1}+2m_2-1),4 ]] $. The encoding rate of the code $ \frac{k}{n}=\frac{m_1(3^{m_1}+2m_2-1)}{2+(m_1-2)(3^{m_1}+2m_2-1)} =$ $ \frac{2}{2+(m_1-2)(3^{m_1}+2m_2-1)} +(1-\frac{2}{m_1})\rightarrow 1$ as $ m_1,m_2\rightarrow \infty $.

As the encoding rate of the quantum codes presented in Section \mbox{\ref{code 2m1-1}} and \mbox{\ref{code 2m1}} is $1$, these codes can therefore be considered as good codes. A comparison of these codes is presented at Sl No 12 and 13 in Table \mbox{\ref{quantum_code}}.
	
	\scriptsize{	
		\begin{eqnarray}
		\setcounter{MaxMatrixCols}{42}
		H_{X} =
		\begin{pmatrix}
		1 ~ 1 ~ 1 ~ 1 ~ 1 ~ 1 ~ 1 ~ 0 ~ 0 ~ 0 ~ 0 ~ 0 ~ 0 ~ 0 ~ 0 ~ 0 ~ 0 ~ 0 ~ 0 ~ 0 ~ 0  ~
		0 ~ 0 ~ 0 ~ 0 ~ 0 ~ 0 ~ 0 ~ 0 ~ 0 ~ 0 ~ 0 ~ 0 ~ 0 ~ 0 ~ 0 ~ 0 ~ 0 ~ 0 ~ 0 ~ 0 ~ 0 \\
		1 ~ 0 ~ 0 ~ 0 ~ 0 ~ 0 ~ 0 ~ 1 ~ 1 ~ 1 ~ 1 ~ 1 ~ 1 ~ 0 ~ 0 ~ 0 ~ 0 ~ 0 ~ 0 ~ 0 ~ 0  ~
		0 ~ 0 ~ 0 ~ 0 ~ 0 ~ 0 ~ 0 ~ 0 ~ 0 ~ 0 ~ 0 ~ 0 ~ 0 ~ 0 ~ 0 ~ 0 ~ 0 ~ 0 ~ 0 ~ 0 ~ 0 \\
		0 ~ 1 ~ 0 ~ 0 ~ 0 ~ 0 ~ 0 ~ 0 ~ 0 ~ 0 ~ 0 ~ 1 ~ 0 ~ 1 ~ 1 ~ 1 ~ 1 ~ 1 ~ 0 ~ 0 ~ 0  ~
		0 ~ 0 ~ 0 ~ 0 ~ 0 ~ 0 ~ 0 ~ 0 ~ 0 ~ 0 ~ 0 ~ 0 ~ 0 ~ 0 ~ 0 ~ 0 ~ 0 ~ 0 ~ 0 ~ 0 ~ 0 \\
		0 ~ 0 ~ 0 ~ 0 ~ 0 ~ 0 ~ 0 ~ 0 ~ 0 ~ 0 ~ 0 ~ 0 ~ 1 ~ 0 ~ 0 ~ 0 ~ 0 ~ 1 ~ 1 ~ 1 ~ 1  ~
		1 ~ 1 ~ 0 ~ 0 ~ 0 ~ 0 ~ 0 ~ 0 ~ 0 ~ 0 ~ 0 ~ 0 ~ 0 ~ 0 ~ 0 ~ 0 ~ 0 ~ 0 ~ 0 ~ 0 ~ 0 \\
		0 ~ 0 ~ 1 ~ 0 ~ 0 ~ 0 ~ 0 ~ 0 ~ 0 ~ 0 ~ 0 ~ 0 ~ 0 ~ 0 ~ 0 ~ 0 ~ 1 ~ 1 ~ 1 ~ 0 ~ 0  ~
		0 ~ 0 ~ 1 ~ 1 ~ 1 ~ 1 ~ 0 ~ 0 ~ 0 ~ 0 ~ 0 ~ 0 ~ 0 ~ 0 ~ 0 ~ 0 ~ 0 ~ 0 ~ 0 ~ 0 ~ 0 \\
		0 ~ 0 ~ 0 ~ 0 ~ 0 ~ 0 ~ 0 ~ 0 ~ 0 ~ 1 ~ 0 ~ 0 ~ 0 ~ 0 ~ 0 ~ 0 ~ 0 ~ 0 ~ 0 ~ 0 ~ 0 ~
		0 ~ 1 ~ 0 ~ 0 ~ 0 ~ 1 ~ 1 ~ 1 ~ 1 ~ 1 ~ 0 ~ 0 ~ 0 ~ 0 ~ 0 ~ 0 ~ 0 ~ 0 ~ 0 ~ 0 ~ 0 \\
		0 ~ 0 ~ 0 ~ 1 ~ 0 ~ 0 ~ 0 ~ 0 ~ 0 ~ 0 ~ 0 ~ 0 ~ 0 ~ 1 ~ 0 ~ 0 ~ 0 ~ 0 ~ 0 ~ 0 ~ 0 ~
		0 ~ 0 ~ 0 ~ 0 ~ 1 ~ 0 ~ 1 ~ 0 ~ 0 ~ 0 ~ 1 ~ 1 ~ 1 ~ 0 ~ 0 ~ 0 ~ 0 ~ 0 ~ 0 ~ 0 ~ 0 \\
		0 ~ 0 ~ 0 ~ 0 ~ 0 ~ 0 ~ 0 ~ 1 ~ 0 ~ 0 ~ 0 ~ 0 ~ 0 ~ 0 ~ 0 ~ 0 ~ 0 ~ 0 ~ 0 ~ 1 ~ 0 ~
		0 ~ 0 ~ 0 ~ 0 ~ 0 ~ 0 ~ 0 ~ 1 ~ 0 ~ 0 ~ 0 ~ 1 ~ 0 ~ 1 ~ 1 ~ 1 ~ 0 ~ 0 ~ 0 ~ 0 ~ 0 \\
		0 ~ 0 ~ 0 ~ 0 ~ 1 ~ 0 ~ 0 ~ 0 ~ 0 ~ 0 ~ 0 ~ 0 ~ 0 ~ 0 ~ 1 ~ 0 ~ 0 ~ 0 ~ 0 ~ 0 ~ 0 ~
		0 ~ 0 ~ 1 ~ 0 ~ 0 ~ 0 ~ 0 ~ 0 ~ 0 ~ 0 ~ 0 ~ 0 ~ 1 ~ 0 ~ 0 ~ 1 ~ 1 ~ 1 ~ 0 ~ 0 ~ 0 \\
		0 ~ 0 ~ 0 ~ 0 ~ 0 ~ 0 ~ 0 ~ 0 ~ 1 ~ 0 ~ 0 ~ 0 ~ 0 ~ 0 ~ 0 ~ 0 ~ 0 ~ 0 ~ 0 ~ 0 ~ 0 ~
		1 ~ 0 ~ 0 ~ 0 ~ 0 ~ 0 ~ 0 ~ 0 ~ 0 ~ 1 ~ 0 ~ 0 ~ 0 ~ 1 ~ 0 ~ 0 ~ 1 ~ 0 ~ 1 ~ 1 ~ 0 \\
		0 ~ 0 ~ 0 ~ 0 ~ 0 ~ 1 ~ 0 ~ 0 ~ 0 ~ 0 ~ 0 ~ 0 ~ 0 ~ 0 ~ 0 ~ 1 ~ 0 ~ 0 ~ 0 ~ 0 ~ 0 ~
		0 ~ 0 ~ 0 ~ 1 ~ 0 ~ 0 ~ 0 ~ 0 ~ 0 ~ 0 ~ 1 ~ 0 ~ 0 ~ 0 ~ 0 ~ 0 ~ 0 ~ 1 ~ 1 ~ 0 ~ 1 \\
		0 ~ 0 ~ 0 ~ 0 ~ 0 ~ 0 ~ 1 ~ 0 ~ 0 ~ 0 ~ 1 ~ 0 ~ 0 ~ 0 ~ 0 ~ 0 ~ 0 ~ 0 ~ 0 ~ 0 ~ 1 ~
		0 ~ 0 ~ 0 ~ 0 ~ 0 ~ 0 ~ 0 ~ 0 ~ 1 ~ 0 ~ 0 ~ 0 ~ 0 ~ 0 ~ 1 ~ 0 ~ 0 ~ 0 ~ 0 ~ 1 ~ 1 \\
		\end{pmatrix}_{12\times 42} \label{hx for 37}
		\end{eqnarray}}
		%
		\scriptsize{
		\begin{eqnarray}
		\setcounter{MaxMatrixCols}{42}
		H_{Z} =
		\begin{pmatrix}
		0 ~ 0 ~ 0 ~ 0 ~ 0 ~ 0 ~ 0 ~ 0 ~ 0 ~ 0 ~ 0 ~ 0 ~ 0 ~ 0 ~ 0 ~ 0 ~ 0 ~ 0 ~ 0 ~ 0 ~ 0 ~
		0 ~ 0 ~ 0 ~ 0 ~ 0 ~ 0 ~ 0 ~ 0 ~ 0 ~ 0 ~ 0 ~ 0 ~ 0 ~ 0 ~ 0 ~ 0 ~ 0 ~ 0 ~ 1 ~ 1 ~ 1 \\
		0 ~ 0 ~ 0 ~ 0 ~ 0 ~ 0 ~ 0 ~ 0 ~ 0 ~ 0 ~ 0 ~ 0 ~ 0 ~ 0 ~ 0 ~ 0 ~ 0 ~ 0 ~ 0 ~ 0 ~ 1 ~
		1 ~ 0 ~ 0 ~ 0 ~ 0 ~ 0 ~ 0 ~ 0 ~ 0 ~ 0 ~ 0 ~ 0 ~ 0 ~ 0 ~ 0 ~ 0 ~ 0 ~ 0 ~ 0 ~ 1 ~ 0 \\
		0 ~ 0 ~ 0 ~ 0 ~ 0 ~ 0 ~ 0 ~ 0 ~ 0 ~ 0 ~ 0 ~ 1 ~ 1 ~ 0 ~ 0 ~ 0 ~ 0 ~ 1 ~ 0 ~ 0 ~ 0 ~
		0 ~ 0 ~ 0 ~ 0 ~ 0 ~ 0 ~ 0 ~ 0 ~ 0 ~ 0 ~ 0 ~ 0 ~ 0 ~ 0 ~ 0 ~ 0 ~ 0 ~ 0 ~ 0 ~ 0 ~ 0 \\
		0 ~ 0 ~ 0 ~ 0 ~ 0 ~ 0 ~ 0 ~ 1 ~ 0 ~ 0 ~ 0 ~ 0 ~ 1 ~ 0 ~ 0 ~ 0 ~ 0 ~ 0 ~ 0 ~ 1 ~ 0 ~
		0 ~ 0 ~ 0 ~ 0 ~ 0 ~ 0 ~ 0 ~ 0 ~ 0 ~ 0 ~ 0 ~ 0 ~ 0 ~ 0 ~ 0 ~ 0 ~ 0 ~ 0 ~ 0 ~ 0 ~ 0 \\
		0 ~ 0 ~ 0 ~ 0 ~ 0 ~ 0 ~ 0 ~ 0 ~ 0 ~ 0 ~ 0 ~ 0 ~ 0 ~ 0 ~ 0 ~ 0 ~ 0 ~ 0 ~ 0 ~ 1 ~ 1 ~
		0 ~ 0 ~ 0 ~ 0 ~ 0 ~ 0 ~ 0 ~ 0 ~ 0 ~ 0 ~ 0 ~ 0 ~ 0 ~ 0 ~ 1 ~ 0 ~ 0 ~ 0 ~ 0 ~ 0 ~ 0 \\
		0 ~ 0 ~ 0 ~ 0 ~ 0 ~ 0 ~ 0 ~ 0 ~ 0 ~ 0 ~ 0 ~ 0 ~ 0 ~ 0 ~ 0 ~ 0 ~ 0 ~ 0 ~ 0 ~ 0 ~ 0 ~
		0 ~ 0 ~ 0 ~ 0 ~ 0 ~ 0 ~ 0 ~ 1 ~ 1 ~ 0 ~ 0 ~ 0 ~ 0 ~ 0 ~ 1 ~ 0 ~ 0 ~ 0 ~ 0 ~ 0 ~ 0 \\
		0 ~ 0 ~ 0 ~ 0 ~ 0 ~ 0 ~ 0 ~ 0 ~ 0 ~ 0 ~ 0 ~ 0 ~ 0 ~ 0 ~ 0 ~ 0 ~ 0 ~ 0 ~ 0 ~ 0 ~ 0 ~
		0 ~ 0 ~ 0 ~ 0 ~ 0 ~ 0 ~ 0 ~ 0 ~ 0 ~ 0 ~ 0 ~ 1 ~ 1 ~ 0 ~ 0 ~ 1 ~ 0 ~ 0 ~ 0 ~ 0 ~ 0 \\
		0 ~ 0 ~ 0 ~ 0 ~ 0 ~ 0 ~ 0 ~ 0 ~ 0 ~ 0 ~ 0 ~ 0 ~ 0 ~ 0 ~ 0 ~ 0 ~ 0 ~ 0 ~ 0 ~ 0 ~ 0 ~
		0 ~ 0 ~ 0 ~ 0 ~ 0 ~ 0 ~ 1 ~ 1 ~ 0 ~ 0 ~ 0 ~ 1 ~ 0 ~ 0 ~ 0 ~ 0 ~ 0 ~ 0 ~ 0 ~ 0 ~ 0 \\
		0 ~ 0 ~ 0 ~ 0 ~ 0 ~ 0 ~ 0 ~ 0 ~ 0 ~ 0 ~ 0 ~ 0 ~ 0 ~ 0 ~ 0 ~ 0 ~ 0 ~ 0 ~ 0 ~ 0 ~ 0 ~
		0 ~ 0 ~ 0 ~ 0 ~ 1 ~ 1 ~ 1 ~ 0 ~ 0 ~ 0 ~ 0 ~ 0 ~ 0 ~ 0 ~ 0 ~ 0 ~ 0 ~ 0 ~ 0 ~ 0 ~ 0 \\
		0 ~ 0 ~ 0 ~ 0 ~ 0 ~ 0 ~ 0 ~ 0 ~ 0 ~ 0 ~ 0 ~ 0 ~ 0 ~ 0 ~ 0 ~ 0 ~ 0 ~ 0 ~ 0 ~ 0 ~ 0 ~
		0 ~ 0 ~ 0 ~ 1 ~ 1 ~ 0 ~ 0 ~ 0 ~ 0 ~ 0 ~ 1 ~ 0 ~ 0 ~ 0 ~ 0 ~ 0 ~ 0 ~ 0 ~ 0 ~ 0 ~ 0 \\
		0 ~ 0 ~ 0 ~ 0 ~ 0 ~ 0 ~ 0 ~ 0 ~ 0 ~ 0 ~ 0 ~ 0 ~ 0 ~ 0 ~ 0 ~ 0 ~ 0 ~ 0 ~ 0 ~ 0 ~ 0 ~
		0 ~ 0 ~ 0 ~ 0 ~ 0 ~ 0 ~ 0 ~ 0 ~ 0 ~ 0 ~ 0 ~ 0 ~ 0 ~ 0 ~ 0 ~ 0 ~ 1 ~ 1 ~ 1 ~ 0 ~ 0 \\
		0 ~ 0 ~ 0 ~ 0 ~ 0 ~ 0 ~ 0 ~ 0 ~ 0 ~ 0 ~ 0 ~ 0 ~ 0 ~ 0 ~ 1 ~ 1 ~ 0 ~ 0 ~ 0 ~ 0 ~ 0 ~
		0 ~ 0 ~ 0 ~ 0 ~ 0 ~ 0 ~ 0 ~ 0 ~ 0 ~ 0 ~ 0 ~ 0 ~ 0 ~ 0 ~ 0 ~ 0 ~ 0 ~ 1 ~ 0 ~ 0 ~ 0 \\
		0 ~ 0 ~ 0 ~ 0 ~ 0 ~ 0 ~ 0 ~ 0 ~ 0 ~ 0 ~ 0 ~ 0 ~ 0 ~ 1 ~ 0 ~ 1 ~ 0 ~ 0 ~ 0 ~ 0 ~ 0 ~
		0 ~ 0 ~ 0 ~ 0 ~ 0 ~ 0 ~ 0 ~ 0 ~ 0 ~ 0 ~ 1 ~ 0 ~ 0 ~ 0 ~ 0 ~ 0 ~ 0 ~ 0 ~ 0 ~ 0 ~ 0 \\
		0 ~ 1 ~ 0 ~ 1 ~ 0 ~ 0 ~ 0 ~ 0 ~ 0 ~ 0 ~ 0 ~ 0 ~ 0 ~ 1 ~ 0 ~ 0 ~ 0 ~ 0 ~ 0 ~ 0 ~ 0 ~
		0 ~ 0 ~ 0 ~ 0 ~ 0 ~ 0 ~ 0 ~ 0 ~ 0 ~ 0 ~ 0 ~ 0 ~ 0 ~ 0 ~ 0 ~ 0 ~ 0 ~ 0 ~ 0 ~ 0 ~ 0 \\
		0 ~ 0 ~ 0 ~ 1 ~ 1 ~ 0 ~ 0 ~ 0 ~ 0 ~ 0 ~ 0 ~ 0 ~ 0 ~ 0 ~ 0 ~ 0 ~ 0 ~ 0 ~ 0 ~ 0 ~ 0 ~
		0 ~ 0 ~ 0 ~ 0 ~ 0 ~ 0 ~ 0 ~ 0 ~ 0 ~ 0 ~ 0 ~ 0 ~ 1 ~ 0 ~ 0 ~ 0 ~ 0 ~ 0 ~ 0 ~ 0 ~ 0 \\
		0 ~ 0 ~ 1 ~ 0 ~ 1 ~ 0 ~ 0 ~ 0 ~ 0 ~ 0 ~ 0 ~ 0 ~ 0 ~ 0 ~ 0 ~ 0 ~ 0 ~ 0 ~ 0 ~ 0 ~ 0 ~
		0 ~ 0 ~ 1 ~ 0 ~ 0 ~ 0 ~ 0 ~ 0 ~ 0 ~ 0 ~ 0 ~ 0 ~ 0 ~ 0 ~ 0 ~ 0 ~ 0 ~ 0 ~ 0 ~ 0 ~ 0 \\
		0 ~ 0 ~ 0 ~ 0 ~ 0 ~ 0 ~ 0 ~ 0 ~ 0 ~ 0 ~ 0 ~ 0 ~ 0 ~ 0 ~ 1 ~ 0 ~ 1 ~ 0 ~ 0 ~ 0 ~ 0 ~
		0 ~ 0 ~ 1 ~ 0 ~ 0 ~ 0 ~ 0 ~ 0 ~ 0 ~ 0 ~ 0 ~ 0 ~ 0 ~ 0 ~ 0 ~ 0 ~ 0 ~ 0 ~ 0 ~ 0 ~ 0 \\
		0 ~ 0 ~ 0 ~ 0 ~ 0 ~ 0 ~ 0 ~ 0 ~ 0 ~ 0 ~ 0 ~ 0 ~ 0 ~ 0 ~ 0 ~ 0 ~ 1 ~ 1 ~ 1 ~ 0 ~ 0 ~
		0 ~ 0 ~ 0 ~ 0 ~ 0 ~ 0 ~ 0 ~ 0 ~ 0 ~ 0 ~ 0 ~ 0 ~ 0 ~ 0 ~ 0 ~ 0 ~ 0 ~ 0 ~ 0 ~ 0 ~ 0 \\
		0 ~ 0 ~ 1 ~ 0 ~ 0 ~ 1 ~ 0 ~ 0 ~ 0 ~ 0 ~ 0 ~ 0 ~ 0 ~ 0 ~ 0 ~ 0 ~ 0 ~ 0 ~ 0 ~ 0 ~ 0 ~
		0 ~ 0 ~ 0 ~ 1 ~ 0 ~ 0 ~ 0 ~ 0 ~ 0 ~ 0 ~ 0 ~ 0 ~ 0 ~ 0 ~ 0 ~ 0 ~ 0 ~ 0 ~ 0 ~ 0 ~ 0 \\
		0 ~ 0 ~ 0 ~ 0 ~ 0 ~ 1 ~ 1 ~ 0 ~ 0 ~ 0 ~ 0 ~ 0 ~ 0 ~ 0 ~ 0 ~ 0 ~ 0 ~ 0 ~ 0 ~ 0 ~ 0 ~
		0 ~ 0 ~ 0 ~ 0 ~ 0 ~ 0 ~ 0 ~ 0 ~ 0 ~ 0 ~ 0 ~ 0 ~ 0 ~ 0 ~ 0 ~ 0 ~ 0 ~ 0 ~ 0 ~ 0 ~ 1 \\
		1 ~ 0 ~ 0 ~ 0 ~ 0 ~ 0 ~ 1 ~ 0 ~ 0 ~ 0 ~ 1 ~ 0 ~ 0 ~ 0 ~ 0 ~ 0 ~ 0 ~ 0 ~ 0 ~ 0 ~ 0 ~
		0 ~ 0 ~ 0 ~ 0 ~ 0 ~ 0 ~ 0 ~ 0 ~ 0 ~ 0 ~ 0 ~ 0 ~ 0 ~ 0 ~ 0 ~ 0 ~ 0 ~ 0 ~ 0 ~ 0 ~ 0 \\
		1 ~ 1 ~ 0 ~ 0 ~ 0 ~ 0 ~ 0 ~ 0 ~ 0 ~ 0 ~ 1 ~ 0 ~ 0 ~ 0 ~ 0 ~ 0 ~ 0 ~ 0 ~ 0 ~ 0 ~ 0 ~
		0 ~ 0 ~ 0 ~ 0 ~ 0 ~ 0 ~ 0 ~ 0 ~ 0 ~ 0 ~ 0 ~ 0 ~ 0 ~ 0 ~ 0 ~ 0 ~ 0 ~ 0 ~ 0 ~ 0 ~ 0 \\
		0 ~ 0 ~ 0 ~ 0 ~ 0 ~ 0 ~ 0 ~ 0 ~ 1 ~ 0 ~ 1 ~ 0 ~ 0 ~ 0 ~ 0 ~ 0 ~ 0 ~ 0 ~ 0 ~ 0 ~ 0 ~
		0 ~ 0 ~ 0 ~ 0 ~ 0 ~ 0 ~ 0 ~ 0 ~ 1 ~ 0 ~ 0 ~ 0 ~ 0 ~ 0 ~ 0 ~ 0 ~ 0 ~ 0 ~ 0 ~ 0 ~ 0 \\
		0 ~ 0 ~ 0 ~ 0 ~ 0 ~ 0 ~ 0 ~ 0 ~ 1 ~ 1 ~ 0 ~ 0 ~ 0 ~ 0 ~ 0 ~ 0 ~ 0 ~ 0 ~ 0 ~ 0 ~ 0 ~
		0 ~ 0 ~ 0 ~ 0 ~ 0 ~ 0 ~ 0 ~ 0 ~ 0 ~ 1 ~ 0 ~ 0 ~ 0 ~ 0 ~ 0 ~ 0 ~ 0 ~ 0 ~ 0 ~ 0 ~ 0 \\
		0 ~ 0 ~ 0 ~ 0 ~ 0 ~ 0 ~ 0 ~ 1 ~ 1 ~ 0 ~ 0 ~ 0 ~ 0 ~ 0 ~ 0 ~ 0 ~ 0 ~ 0 ~ 0 ~ 0 ~ 0 ~
		0 ~ 0 ~ 0 ~ 0 ~ 0 ~ 0 ~ 0 ~ 0 ~ 0 ~ 0 ~ 0 ~ 0 ~ 0 ~ 1 ~ 0 ~ 0 ~ 0 ~ 0 ~ 0 ~ 0 ~ 0 \\
		0 ~ 0 ~ 0 ~ 0 ~ 0 ~ 0 ~ 0 ~ 0 ~ 0 ~ 0 ~ 0 ~ 0 ~ 0 ~ 0 ~ 0 ~ 0 ~ 0 ~ 0 ~ 0 ~ 0 ~ 0 ~
		0 ~ 0 ~ 0 ~ 0 ~ 0 ~ 0 ~ 0 ~ 0 ~ 0 ~ 0 ~ 0 ~ 0 ~ 0 ~ 1 ~ 0 ~ 1 ~ 1 ~ 0 ~ 0 ~ 0 ~ 0 \\
		0 ~ 0 ~ 0 ~ 0 ~ 0 ~ 0 ~ 0 ~ 0 ~ 0 ~ 0 ~ 0 ~ 0 ~ 0 ~ 0 ~ 0 ~ 0 ~ 0 ~ 0 ~ 0 ~ 0 ~ 0 ~
		1 ~ 1 ~ 0 ~ 0 ~ 0 ~ 0 ~ 0 ~ 0 ~ 0 ~ 1 ~ 0 ~ 0 ~ 0 ~ 0 ~ 0 ~ 0 ~ 0 ~ 0 ~ 0 ~ 0 ~ 0 \\
		0 ~ 0 ~ 0 ~ 0 ~ 0 ~ 0 ~ 0 ~ 0 ~ 0 ~ 0 ~ 0 ~ 0 ~ 0 ~ 0 ~ 0 ~ 0 ~ 0 ~ 0 ~ 1 ~ 0 ~ 0 ~
		0 ~ 1 ~ 0 ~ 0 ~ 0 ~ 1 ~ 0 ~ 0 ~ 0 ~ 0 ~ 0 ~ 0 ~ 0 ~ 0 ~ 0 ~ 0 ~ 0 ~ 0 ~ 0 ~ 0 ~ 0 \\
		\end{pmatrix}_{28\times 42} \label{hz for 37}
		\end{eqnarray}}
\setcounter{MaxMatrixCols}{40}
	\scriptsize{

		\begin{eqnarray}
		H_X=\setlength\arraycolsep{1pt}
		\begin{pmatrix}
		1	&1	&1	&1	&0	&0	&0	&0	&0	&0	&0	&0	&0	&0	&0	&0	&0	&0	&0	&0	&0	&0	&0	&0	&0	&0	&0	&0	&0	&0	&0	&0	&0	&0	&0	&0	&0	&0	&0	&0\\
		1	&0	&0	&0	&1	&1	&1	&0	&0	&0	&0	&0	&0	&0	&0	&0	&0	&0	&0	&0	&0	&0	&0	&0	&0	&0	&0	&0	&0	&0	&0	&0	&0	&0	&0	&0	&0	&0	&0	&0\\
		0	&0	&0	&0	&1	&0	&0	&1	&1	&1	&0	&0	&0	&0	&0	&0	&0	&0	&0	&0	&0	&0	&0	&0	&0	&0	&0	&0	&0	&0	&0	&0	&0	&0	&0	&0	&0	&0	&0	&0\\
		0	&1	&0	&0	&0	&0	&0	&1	&0	&0	&1	&1	&0	&0	&0	&0	&0	&0	&0	&0	&0	&0	&0	&0	&0	&0	&0	&0	&0	&0	&0	&0	&0	&0	&0	&0	&0	&0	&0	&0\\
		0	&0	&0	&0	&0	&0	&0	&0	&0	&0	&1	&0	&1	&1	&1	&0	&0	&0	&0	&0	&0	&0	&0	&0	&0	&0	&0	&0	&0	&0	&0	&0	&0	&0	&0	&0	&0	&0	&0	&0\\
		0	&0	&1	&0	&0	&0	&0	&0	&0	&0	&0	&0	&1	&0	&0	&1	&1	&0	&0	&0	&0	&0	&0	&0	&0	&0	&0	&0	&0	&0	&0	&0	&0	&0	&0	&0	&0	&0	&0	&0\\
		0	&0	&0	&0	&0	&0	&0	&0	&0	&0	&0	&0	&0	&0	&0	&1	&0	&1	&1	&1	&0	&0	&0	&0	&0	&0	&0	&0	&0	&0	&0	&0	&0	&0	&0	&0	&0	&0	&0	&0\\
		0	&0	&0	&1	&0	&0	&0	&0	&0	&0	&0	&0	&0	&0	&0	&0	&0	&1	&0	&0	&1	&1	&0	&0	&0	&0	&0	&0	&0	&0	&0	&0	&0	&0	&0	&0	&0	&0	&0	&0\\
		0	&0	&0	&0	&0	&0	&0	&0	&0	&0	&0	&0	&0	&0	&0	&0	&0	&0	&0	&0	&1	&0	&1	&1	&1	&0	&0	&0	&0	&0	&0	&0	&0	&0	&0	&0	&0	&0	&0	&0\\
		0	&0	&0	&0	&0	&0	&0	&0	&0	&1	&0	&0	&0	&0	&0	&0	&0	&0	&0	&0	&0	&0	&1	&0	&0	&1	&1	&0	&0	&0	&0	&0	&0	&0	&0	&0	&0	&0	&0	&0\\
		0	&0	&0	&0	&0	&0	&1	&0	&0	&0	&0	&0	&0	&1	&0	&0	&0	&0	&0	&0	&0	&0	&0	&0	&0	&0	&0	&1	&1	&0	&0	&0	&0	&0	&0	&0	&0	&0	&0	&0\\
		0	&0	&0	&0	&0	&0	&0	&0	&1	&0	&0	&0	&0	&0	&0	&0	&1	&0	&0	&0	&0	&0	&0	&0	&0	&0	&0	&1	&0	&1	&0	&0	&0	&0	&0	&0	&0	&0	&0	&0\\
		0	&0	&0	&0	&0	&0	&0	&0	&0	&0	&0	&0	&0	&0	&0	&0	&0	&0	&0	&0	&0	&0	&0	&0	&0	&1	&0	&0	&1	&0	&1	&1	&0	&0	&0	&0	&0	&0	&0	&0\\
		0	&0	&0	&0	&0	&0	&0	&0	&0	&0	&0	&0	&0	&0	&1	&0	&0	&0	&0	&0	&0	&1	&0	&0	&0	&0	&0	&0	&0	&0	&0	&0	&1	&0	&0	&0	&0	&0	&0	&1\\
		0	&0	&0	&0	&0	&0	&0	&0	&0	&0	&0	&0	&0	&0	&0	&0	&0	&0	&0	&0	&0	&0	&0	&1	&0	&0	&0	&0	&0	&0	&1	&0	&1	&1	&0	&0	&0	&0	&0	&0\\
		0	&0	&0	&0	&0	&0	&0	&0	&0	&1	&0	&0	&0	&0	&0	&0	&0	&0	&0	&0	&0	&0	&0	&0	&0	&0	&1	&0	&0	&0	&0	&0	&0	&0	&1	&1	&0	&0	&0	&0\\
		0	&0	&0	&0	&0	&0	&0	&0	&0	&0	&0	&0	&0	&0	&0	&0	&0	&0	&0	&0	&0	&0	&0	&0	&1	&0	&0	&0	&0	&0	&0	&0	&0	&0	&0	&1	&1	&1	&0	&0\\
		0	&0	&0	&0	&0	&0	&0	&0	&0	&0	&0	&0	&0	&0	&0	&0	&0	&0	&0	&0	&0	&0	&0	&0	&0	&0	&0	&0	&0	&1	&0	&1	&0	&0	&1	&0	&0	&0	&1	&0\\
		0	&0	&0	&0	&0	&0	&0	&0	&0	&0	&0	&1	&0	&0	&0	&0	&0	&0	&1	&0	&0	&0	&0	&0	&0	&0	&0	&0	&0	&0	&0	&0	&0	&0	&0	&0	&1	&0	&0	&1\\
		0	&0	&0	&0	&0	&0	&0	&0	&0	&0	&0	&0	&0	&0	&0	&0	&0	&0	&0	&1	&0	&0	&0	&0	&0	&0	&0	&0	&0	&0	&0	&0	&0	&1	&0	&0	&0	&1	&1	&0
		\end{pmatrix} \label{hx_sem}
		\end{eqnarray}}
	\scriptsize{
		\begin{eqnarray}
		H_Z=
		\setlength\arraycolsep{1pt}
		\begin{pmatrix}
		1 & 1 & 0 & 0 & 1 & 0 & 0 & 1 & 0 & 0 & 0 & 0 & 0 & 0 & 0 & 0 & 0 & 0 & 0 & 0 & 0 & 0 & 0 & 0 & 0 & 0 & 0 & 0 & 0 & 0 & 0 & 0 & 0 & 0 & 0 & 0 & 0 & 0 & 0 & 0\\
		1	&0	&0	&1	&0	&1	&0	&0	&0	&0	&0	&0	&0	&0	&0	&0	&0	&0	&0	&0	&1	&0	&1	&0	&0	&0	&0	&0	&0	&0	&0	&0	&0	&0	&0	&0	&0	&0	&0	&0\\
		0	&1	&1	&0	&0	&0	&0	&0	&0	&0	&1	&0	&1	&0	&0	&0	&0	&0	&0	&0	&0	&0	&0	&0	&0	&0	&0	&0	&0	&0	&0	&0	&0	&0	&0	&0	&0	&0	&0	&0\\
		0	&0	&1	&1	&0	&0	&0	&0	&0	&0	&0	&0	&0	&0	&0	&1	&0	&1	&0	&0	&0	&0	&0	&0	&0	&0	&0	&0	&0	&0	&0	&0	&0	&0	&0	&0	&0	&0	&0	&0\\
		0	&0	&0	&0	&1	&0	&1	&0	&1	&0	&0	&0	&0	&0	&0	&0	&0	&0	&0	&0	&0	&0	&0	&0	&0	&0	&0	&1	&0	&0	&0	&0	&0	&0	&0	&0	&0	&0	&0	&0\\
		0	&0	&0	&0	&0	&1	&1	&0	&0	&0	&0	&0	&0	&0	&0	&0	&0	&0	&0	&0	&0	&0	&0	&0	&0	&1	&0	&0	&1	&0	&0	&0	&0	&0	&0	&0	&0	&0	&0	&0\\
		0	&0	&0	&0	&0	&0	&0	&1	&0	&1	&0	&1	&0	&0	&0	&0	&0	&0	&0	&0	&0	&0	&0	&0	&0	&0	&0	&0	&0	&0	&0	&0	&0	&0	&0	&1	&1	&0	&0	&0\\
		0	&0	&0	&0	&0	&0	&0	&0	&1	&1	&0	&0	&0	&0	&0	&0	&0	&0	&0	&0	&0	&0	&0	&0	&0	&0	&0	&0	&0	&1	&0	&0	&0	&0	&1	&0	&0	&0	&0	&0\\
		0	&0	&0	&0	&0	&0	&0	&0	&0	&0	&1	&1	&0	&0	&1	&0	&0	&0	&0	&0	&1	&0	&1	&1	&1	&0	&0	&0	&0	&0	&0	&0	&0	&0	&0	&0	&0	&0	&0	&1\\
		0	&0	&0	&0	&0	&0	&0	&0	&0	&0	&0	&0	&1	&1	&0	&0	&1	&0	&0	&0	&0	&0	&0	&0	&0	&0	&0	&1	&0	&0	&0	&0	&0	&0	&0	&0	&0	&0	&0	&0\\
		0	&0	&0	&0	&0	&0	&0	&0	&0	&0	&0	&0	&0	&1	&1	&0	&0	&0	&0	&0	&0	&0	&0	&0	&0	&0	&0	&0	&1	&0	&1	&0	&1	&0	&0	&0	&0	&0	&0	&0\\
		0	&0	&0	&0	&0	&0	&0	&0	&0	&0	&0	&0	&0	&0	&0	&1	&1	&0	&0	&1	&0	&0	&0	&0	&0	&0	&0	&0	&0	&1	&0	&0	&0	&0	&0	&0	&0	&0	&1	&0\\
		0	&0	&0	&0	&0	&0	&0	&0	&0	&0	&0	&0	&0	&0	&0	&0	&0	&1	&1	&0	&0	&1	&0	&0	&0	&0	&0	&0	&0	&0	&0	&0	&0	&0	&0	&0	&0	&0	&0	&1\\
		0	&0	&0	&0	&0	&0	&0	&0	&0	&0	&0	&0	&0	&0	&0	&0	&0	&0	&1	&1	&0	&0	&0	&0	&0	&0	&0	&0	&0	&0	&0	&0	&0	&0	&0	&0	&1	&1	&0	&0\\
		0	&0	&0	&0	&0	&0	&0	&0	&0	&0	&0	&0	&0	&0	&0	&0	&0	&0	&0	&0	&1	&1	&0	&1	&0	&0	&0	&0	&0	&0	&0	&0	&1	&0	&0	&0	&0	&0	&0	&0\\
		0	&0	&0	&0	&0	&0	&0	&0	&0	&0	&0	&0	&0	&0	&0	&0	&0	&0	&0	&0	&0	&0	&0	&1	&1	&0	&0	&0	&0	&0	&0	&0	&0	&1	&0	&0	&0	&1	&0	&0\\
		0	&0	&0	&0	&0	&0	&0	&0	&0	&0	&0	&0	&0	&0	&0	&0	&0	&0	&0	&0	&0	&0	&1	&0	&1	&0	&1	&0	&0	&0	&0	&0	&0	&0	&0	&1	&0	&0	&0	&0\\
		0	&0	&0	&0	&0	&0	&0	&0	&0	&0	&0	&0	&0	&0	&0	&0	&0	&0	&0	&0	&0	&0	&0	&0	&0	&1	&1	&0	&0	&0	&0	&1	&0	&0	&1	&0	&0	&0	&0	&0\\
		0	&0	&0	&0	&0	&0	&0	&0	&0	&0	&0	&0	&0	&0	&0	&0	&0	&0	&0	&0	&0	&0	&0	&0	&0	&0	&0	&0	&0	&0	&1	&1	&0	&1	&0	&0	&0	&0	&1	&0
		\end{pmatrix} \label{hz_sem}
		\end{eqnarray}}
	
\normalsize

\section{Examples and codes associated with \lowercase{d}-th cover maps} \label{37 quantum code}
We know from \cite{du2006} that equivelar maps of type $ [p^q] $ exist for $ p = 3, q = 7 $ on a double torus with $ 12 $ vertices: namely, $ \mathcal{N}_i $, $ i= 1, 2, 3, 4, 5, 6$. {In particular, $ \mathcal{N}_1 $=[ [1, 2, 3], [1, 2, 4], [1, 3, 5], [1, 4, 6],  [1, 5, 7], [1, 6, 8], [1, 7, 8], [2, 3, 6], [2, 4, 7], [2, 6, 9], [2, 7, 10], [2, 9, 10], [3, 5, 9], [3, 6, 11], [3, 9, 12], [3, 11, 12], [4, 6, 9], [4, 7, 8], [4, 8, 12], [4, 9, 12], [5, 7, 11], [5, 9, 10], [5, 10, 12], [5, 11, 12], [6, 8, 11], [7, 10, 11], [8, 10, 11], [8, 10, 12] ].}

For every $ \mathcal{N}_i $, it is obvious to envision that each of these maps has a non-contractible cycle of length three, and therefore, for each of them, $ d_{min} = 3$. Now, $ H_X $ and $ H_Z $ corresponding to $ \mathcal{N}_1 $ are given by Eq. (\ref{hx for 37}) and Eq. (\ref{hz for 37}), respectively. Using graph rotation, mentioned in \cite{ringel1974}, we have $ H_XH_Z^T=0 $. Thus, we can obtain a code corresponding to this. Observe that in this case $ n=|f_1|=42 $, $ k=2-\chi(\mathcal{N}_1)=4 $, and therefore the $ [[42, 4, 3]] $ {quantum} code is obtained. Now, for the $ d $-th cover ($ d\geq1 $) of each map $ \mathcal{N}_i $, map $ \mathcal{N}_i^d $ is produced on the orientable surface with $ 12d $ vertices, $ 42d $ edges, and {\Echar{}} $ -2d $. Therefore, $ k=2-$ {${\chi(\mathcal{N}_i)}$} $=2 \times (1+d) $, and hence, a class of {quantum} codes with parameters {$[[n,k,d_{min}]]=$}$ [[42d,2(1+d),3]] $ associated with $ \mathcal{N}_i^d $ is produced.

In \cite{bmu2020}, it is shown that there are {at least 17 SEMs with Euler characteristic $ -1 $, denoted by $ \mathcal{K}_i $ ($ 1 \le i \le 17$).} It is easy to verify that every map $ \mathcal{K}_i $ satisfy the Euler characteristic equation.  

{For example, $ \mathcal{K}_{3} $=[ [1, 2, 10, 9, 8],  [3, 4, 19, 18, 16],  [5, 11, 13, 15, 14],  [6, 7, 20, 18, 12],  [1, 2, 3, 4],  [1, 4, 5, 6],  [1, 6, 7, 8],  [2, 3, 12, 11],  [2, 10, 13, 11],  [3, 12, 18, 16],  [4, 5, 14, 19],  [5, 6, 12, 11],  [7, 8, 14, 19],  [7, 19, 17, 20],  [8, 9, 15, 14],  [9, 10, 16, 17],  [9, 15, 20, 17],  [10, 13, 18, 16],  [13, 15, 20, 18] ]}

In $ \mathcal{K}_3 $, $ |f_0|=20 $, $ |f_1|=40 $ and $ |f_2|=19 $. Thus, the corresponding vertex-edge and face-edge incidence matrices $ H_X $ and $ H_Z $ are of size $ 20\times 40 $ and $ 19\times 40 $, respectively, as given in Eq. (\ref{hx_sem}) and Eq. (\ref{hz_sem}).

Using graph rotation as in the previous case, we conclude that $ H_XH_Z^T=0 $. Therefore, we obtain a $ [[n,k,d_{min}]] $ {quantum} code with $ n=|f_1|=40 $, $ k=2- \chi(\mathcal{K}_3) =3$, and $ d_{min}= 4$, and the length of the shortest non-contractible cycle in $ \mathcal{K}_3 $ and $ \mathcal{K}_3^* $ $, i.e., $ a  $ [[40,3,4]] $ {quantum} code is obtained. Similarly, codes related to maps $ \mathcal{K}_i $, $ i=1,2,\dots, 17 $ can {be} obtained and are listed in Table \ref{code table chi -1}.

As in the case of equivelar maps on a double torus, the $ d $-th $ (d\geq 1) $ covering of {SEMs} $ \mathcal{K}_i $, listed in Table \ref{code table chi -1}, produces a map $ \mathcal{K}_i^d $ of \Echar{}$ -d $, and therefore, for every $ \mathcal{K}_i^d $, $ k = 2 + d $. In $ \mathcal{K}_i^d $, $ i=1,2,3 $, $ (|f_0|,|f_1|)=(20d,40d) $. Similarly, {$(|f_0|,|f_1|)=(42d,84d)$ for $\mathcal{K}_4$; $(|f_0|,|f_1|)=(42d,63d)$ for $\mathcal{K}_5$; $(|f_0|,|f_1|)=(84d,126d)$ for $\mathcal{K}_6$; $(|f_0|,|f_1|)=(40d,60d)$ for $\mathcal{K}_7$;} $ (|f_0|,|f_1|)=(48d,72d) $ for $ \mathcal{K}_i^d $, {$ i=8,9 $}{;} $ (|f_0|,|f_1|)=(24d,48d) $ for $ \mathcal{K}_i^d $, {$ i=10,11 $}{;} $ (|f_0|,|f_1|)=(24d,36d) $ for $ \mathcal{K}_i^d $, {$ i=12,13 $}{;} $ (|f_0|,|f_1|)=(12d,30d) $ for $ \mathcal{K}_i^d $, {$ i=14 $}{;} and $ (|f_0|,|f_1|)=(12d,36d) $ for $ \mathcal{K}_i^d $, {$ i=15,16,17 $}. Therefore, the code parameter{s are} $ [[n,k,d_{min}]]= [[40d,2+d,4]] $ associated with the map $ \mathcal{K}_i^d $ for $ i=1,2,3 $, {$[[n,k,d_{min}]]=[[84d,2+d,4]]$ associated with the map $\mathcal{K}_4^d$, $[[n,k,d_{min}]]=[[63d,2+d,4]]$ associated with the map $\mathcal{K}_5^d$, $[[n,k,d_{min}]]=[[126d,2+d,4]]$ associated with the map $\mathcal{K}_6^d$, $[[n,k,d_{min}]]=[[60d,2+d,4]]$ associated with the map $\mathcal{K}_7^d$} $ [[n,k,d_{min}]]= [[72d,2+d,4]] $ associated with the map $ \mathcal{K}_i^d $ for {$ i=8,9 $}, $ [[n,k,d_{min}]]= [[48d,2+d,4]]$ associated with the map $ \mathcal{K}_i^d $ for {$ i=10,11 $}, $ [[n,k,d_{min}]]=[[36d,2+d,3]] $ associated with the map $ \mathcal{K}_i^d $ for {$ i=12,13 $}, $ [[n,k,d_{min}]]= [[30d,2+d,4]] $ associated with the map $ \mathcal{K}_i^d $ for $ i=14 $ and $ [[n,k,d_{min}]]= [[36d,2+d,3]]$ associated with the map $ \mathcal{K}_i^d $ for {$ i=15,16,17 $}. The above details are presented in compact form in Table \ref{code for equivelar and semi-equivelar d cover}.
\section{Table of Quantum Codes} \label{table of quantum codes}
\begin{table*}[h!]
	
	\centering
	\begin{tabular}{|p{2.5cm}|p{2.5cm}|p{1.5cm}|p{2cm}|p{1cm}|p{2.5cm}|}
		\hline\hline
		Map Type & Maps & $ n=|f_1| $ & $ k=2-\chi$ & $ d_{min} $ & $ [[n,k,d_{min}]] $\\
		\hline\hline
		$ [4^3,5^1] $ & $ \mathcal{K}_1, \mathcal{K}_2,\mathcal{K}_3 $ & $ 40 $ & $ 3 $ & $ 4 $ & $ [[40,3,4]] $\\
		{$[3^1,4^1,7^1,4^1]$} & {$\mathcal{K}_4$} & {$84$} & {$3$} & {$4$} & {$[[84,3,4]]$}\\
		{$[6^2,7^1]$} & {$\mathcal{K}_5$} & {$63$} & {$3$} & {$4$} & {$[[63,3,4]]$}\\
		{$ [4^1,6^1,14^1] $} & {$\mathcal{K}_6$} & {$126$} & {$3$} & {$4$} & {$[[126,3,4]]$}\\
		{$[4^1,8^1,10^1]$} & {$\mathcal{K}_7$} & {$60$} & {$3$} & {$4$} & {$[[60,3,4]]$}\\
		$ [4^1,6^1,16^1] $ & {$ \mathcal{K}_8, \mathcal{K}_9 $} & $ 72 $ & $ 3 $ & $ 4 $ & $ [[72,3,4]]  $\\
		$ [3^1,4^1,8^1,4^1] $ & {$ \mathcal{K}_{10}, \mathcal{K}_{11} $} & $ 48 $ & $ 3 $ & $ 4 $ & $ [[48,3,4]] $ \\
		$ [6^2,8^1] $ &  {$ \mathcal{K}_{12}, \mathcal{K}_{13} $} & $ 36 $ & $ 3 $ & $ 3 $ & $ [[36,3,3]] $ \\
		$ [3^1,4^1,3^1,4^2] $ & {$ \mathcal{K}_{14} $} & $ 30 $ & $ 3 $ & $ 4 $ & $ [[30,3,4]] $ \\
		$ [3^5,4^1] $ & {$ \mathcal{K}_{15}, \mathcal{K}_{16}, \mathcal{K}_{17} $} & $ 36 $ & $ 3 $ & $ 3 $ & $ [[36,3,3]] $\\
		\hline\hline
	\end{tabular}
	\caption{\small{Table of quantum codes associated with semi-equivelar maps on the surface of $\chi = -1 $}}
	\label{code table chi -1}
\end{table*}
\begin{table*}[h!]
	\centering
	\begin{tabular}{|p{0.5cm}|p{2cm}|p{3cm}|p{0.7cm}|p{7cm}|}
		\hline\hline
		Sl. No& $ n $ & $ k $ & $ d_{min} $ & $ [[n,k,d_{min}]] $ \\
		\hline\hline
		1&$ 2t^2 $ & $ 2 $ & $ t $ & $ [[2t^2,2,t]] $, $ t\geq 3 $ \\
		2&$ \frac{3}{2}t^2 $ & $ 2 $ & $ t $ & $ [[\frac{3}{2}t^2,2,t]] $, $ t=2m(m\geq 2) $ \\
		3&$ t^2 $ & $ 2 $ & $ t $ & $ [[t^2,2,t]] $, $ t=2m(m\geq 2) $\\
		4&$ t^2+1 $ & $ 2 $ & $ t $ & $ [[t^2+1,2,t]] $, $ t=2m+1(m\geq1) $ \\
		5&$ \binom{t}{2} $ & $ \binom{t}{2}-2(t-1) $ & $ 3 $ & $ [[\binom{t}{2},\binom{t}{2}-2(t-1),3]] $, $ (t\equiv 0\vee 1 \mod 4 )  $ \\
		6&$ \frac{t(t-3)}{2} $ & $ \frac{t(t-3)}{2}-2(t-1) $ & $ 3 $ & \scriptsize{$ [[\frac{t(t-3)}{2}, \frac{t(t-3)}{2}-2(t-1), 3]] $, $ (t\equiv 0 \mod 2, t\geq 8) $} \\
		7&$ t(t-5) $ & $ t(t-5)-2(t-1) $ & $ 3 $ & \scriptsize{$ [[t(t-5), t(t-5)-2(t-1), 3]] $, $ (t\equiv 0 \mod 2, t\geq 8) $}\\
		8&$ \frac{t(t-2)}{4} $ & $ \frac{t(t-2)}{4}-2(t-2) $ & $ 3 $ & \scriptsize{$ [[\frac{t(t-2)}{4},\frac{t(t-2)}{4}-2(t-2),3 ]] $, $ (t\equiv 0 \mod 2, t\geq 10) $} \\
		9&$ \frac{3t}{2} $ & $ \frac{t}{2}-10 $ & $ 3 $ & $ [[\frac{3t}{2},\frac{t}{2}-10, 3 ]] $, $ (t\equiv 0 \mod 4, t\geq 28) $ \\
		10&$ t^2 $ & $ t^2-4t+2 $ & $ 4 $ & $ [[t^2, t^2-4t+2,4]] $, $ (t=4,8,10,12,\dots) $ \\
		11&$ 2t(t-1) $ & $ 2(t^2-3t+1) $ & $ 3 $ & $ [[2t(t-1),2(t^2-3t+1)${,3}$]] $, $ t\geq 3 $ \\
		12&\scriptsize{$ (2m_1-1)(3^{m_1-1}+2m_2-1) $} & \scriptsize{$ 2+(2m_1-5)(3^{m_1-1}+2m_2-1) $} & $ 4 $ & \scriptsize{$  [[(2m_1-1)(3^{m_1-1}+2m_2-1),2+(2m_1-5)(3^{m_1-1}+2m_2-1), 4 ]] $, ($ m_1\geq3, m_2\geq0 $)}\\
		13&\scriptsize{$ m_1(3^{m_1}+2m_2-1) $} & \scriptsize{$ 2+(m_1-2)(3^{m_1}+2m_2-1) $} & $ 4 $ & \scriptsize{$[[m_1(3^{m_1}+2m_2-1), 2+(m_1-2)(3^{m_1}+2m_2-1), 4]] $, ($ m_1\geq3, m_2\geq0 $)}\\
		\hline\hline
	\end{tabular}
	\caption{\small{Quantum code $ [[n,k,d_{min}]] $}}
	\label{quantum_code}
\end{table*}
\begin{table*}[h!]
	\centering
	\begin{tabular}{|p{2.5cm}|p{1.5cm}|p{2cm}|p{1cm}|p{3cm}|p{2cm}|}
		\hline\hline
		Map type & $ n=|f_1| $ & $ k $ & $ d_{min} $ & $ [[n,k,d_{min}]] $ & $ \frac{k}{n} $ as $ n\rightarrow \infty $\\
		\hline\hline
		$ [3^7] $ & $ 42d $ & $ 2(1+d) $ & $ 3 $ & $ [[42d,2(1+d),3]] $ & $ \frac{1}{21} $ \\
		$ [4^3,5^1] $ & $ 40d $ & $ 2+d $ & $ 4 $ & $ [[40d,2+d,4]] $ & $ \frac{1}{40} $ \\
		{$[3^1,4^1,7^1,4^1]$} & {$84d$} & {$2+d$} & {$4$} & {$[[84d,2+d,4]]$} & {$\frac{1}{84}$}\\
		{$[6^2,7^1]$} & {$63d$} & {$2+d$} & {$4$} & {$[[63d,2+d,4]]$} & {$\frac{1}{63}$}\\
		{$[4^1,6^1,14^1]$} & {$126d$} & {$2+d$} & {$4$} & {$[[126d,2+d,4]]$} & {$\frac{1}{126}$}\\
		{$[4^1,8^1,10^1]$} & {$60d$} & {$2+d$} & {$4$} & {$[[60d,2+d,4]]$} & {$\frac{1}{60}$}\\
		$ [[4^1,6^1,16^1]] $ & $ 72d $ & $ 2+d $ & $ 4 $ & $ [[72d,2+d,4]] $ & $ \frac{1}{72} $ \\
		$ [3^1,4^1,8^1,4^1] $ & $ 48d $ & $ 2+d $ & $ 4 $ & $ [[48d,2+d,4]] $ & $ \frac{1}{48} $ \\
		$ [6^2,8] $ & $ 36d $ & $ 2+d $ & $ 3 $ & $ [[36d,2+d,3]] $ & $ \frac{1}{36} $ \\
		$ [3^1,4^1,3^1,4^2] $ & $ 30d $ & $ 2+d $ & $ 4 $ & $ [[30d,2+d,4]] $ & $ \frac{1}{30} $ \\
		$ [3^5,4^1] $ & $ 36d $ & $ 2+d $ & $ 3 $ & $ [[36d,2+d,3]] $ & $ \frac{1}{36} $\\
		\hline\hline
	\end{tabular}
	\caption{\small{Code table associated with {covering maps} of equivelar maps on double torus and semi-equivelar maps on the surface of \Echar{-1}}}
	\label{code for equivelar and semi-equivelar d cover}
\end{table*}

The classes of quantum codes presented in Table \ref{quantum_code} from $ 1-11 $ are from references \cite{avaz2018,leslie2014,silva2010,kitaev2003,bombin2006}, and the codes from $ 12-13 $ are presented in this paper. { The encoding rate of the codes from $ 5-13 $ is 1. Note that the minimum distances of the classes of HQC presented in this article are greater than the minimum distance of the class of HQCs available in the literature (except the class in Table 2, Sl No 10).} Codes in Table \ref{code for equivelar and semi-equivelar d cover} are obtained from the $ d $-th covering map of the equivelar and {SEMs} on the double torus and on the surface of \Echar{-1}. { In these code parameters,} {encoding rate} $ \frac{k}{n}\rightarrow \alpha $ as $ n\rightarrow \infty $ with $ \alpha< 1 $.
\section{Conclusion}\label{conclution}
{The central idea of TQC is to make the quantum states depend upon topological properties of a physical system, because topological properties are invariant under smooth degradations. Therefore, obtaining new classes of codes means adding new topological properties to the class of existing topological properties where information can be stored. In this paper, we have constructed thirteen new classes of homological quantum codes. These codes are associated with maps, namely, equivelar maps and semi-equivelar maps on some surfaces. In particular, we have introduced two classes of quantum codes associated with the maps of type $ [k^k] $, i.e., self-dual maps. In this case, the encoding rate is $ \frac{k}{n}\rightarrow 1 $ when $ n\rightarrow\infty $. Also, we have presented eleven new classes of quantum codes which are associated with the $d$-th covering of equivelar and semi-equivelar maps of a double torus and the surface with Euler characteristic $-1$. In both cases, the encoding rate $ \frac{k}{n} $ is $ \alpha $, where $ \alpha< 1 $. This work can be extended to other classes of maps on other surfaces such as Catalon maps in the same manner.}



\begin{thebibliography}{00}
	\bibitem{calderbank1998}
	A.R. Calderbank, E. Rains, P.W. Shor, N. Sloane,  {\em Quantum error correction via codes over $ GF(4) $}, IEEE Trans. Inf. Theory   {\bf 44} (1998), 1369-1387.
	
	\bibitem{kitaev2003}
	A.Y. Kitaev,  {\em Fault-tolerant quantum computation by anyons}, Ann. Phys.  {\bf 303} (2003), 2-30.
	
	\bibitem{hatcher}
	Allen Hatcher, {\em Algebraic topology}, Cambridge University Press, 2002.
	
	\bibitem{avaz2018}
	Avaz Naghipour, {\em New classes of quantum codes on closed orientable surfaces}, Cryptography and Communications, https://doi.org/10.1007/s12095-018-0347-9.
	
	\bibitem{tu2018}
	A. K. Tiwari and A. K. Upadhyay, {\em Semi-equivelar maps on the surface of Euler characteristic -1} Note Mat. {\bf 37} (2017), 91--102.
	
	\bibitem{utm2014}
	A. K. Upadhyay, A. K. Tiwari and D. Maity  {\em Semi-equivelar maps}, Beitr\"{a}ge Algebra Geom.  {\bf 55} (2014), 229--242.
	
	\bibitem{du2006}
	B. Datta, A.K. Upadhyay,  {\em Degree-regular triangulations of the double-torus}, Forum Math.  {\bf 18} (2006), 1011--1025.
	
	\bibitem{dm2018}
	B. Datta and D. Maity, {\em Semiequivelar maps on the torus and the Klein bottle are Archimedean}, Discrete Math. 341(12) (2018), 329--3309.
	
	\bibitem{dutta05}
	B. Dutta,  {\em A Note on the Existence of $ \{k, k\} $-equivelar Polyhedral Maps}, Beitr\"{a}ge zur Algebra und Geometrie,  { Volume 46 (2005), No. 2, 537-544}.
	
	\bibitem{silva2009}
	C. D. Albuquerque, R. Palazzo and E. B. Silva, {\em Topological quantum codes on compact surfaces with genus $ g\geq2 $}, J. Math. Phys., {\bf 50}, 023513(2009); https://doi.org/10.1063/1.3081056.
	
	\bibitem{silva2010}
	C. D. Albuquerque, R. Palazzo and E. B. Silva, {\em New classes of topological quantum codes associated with self-dual, quasi self-dual and denser tessellation}, Quantum Information and Computation, Vol. 10, No. 11\&12 (2010) 0956–0970
	
	\bibitem{bmu2020}
	D. Bhowmik, D. Maity, A.K. Upadhyay and B.P. Yadav, {\em {Semi-equivelar maps on the surface of Euler genus 3}}, \url{http://arxiv.org/abs/2002.06367}
	
	\bibitem{bu2019}
	D. Bhowmik and A. K. Upadhyay, {\em A Classification of Semi-equivelar maps on the surface of Euler characteristic -1}, accepted in Indian Journal of Pure and Applied Mathematics, 2020 (May).
	
	\bibitem{ringel1974}
	G. Ringel, {\em Map Color Theorem}, Grundlehren Der Mathematischen Wissenschaften Bd. 209. Springer, New York (1974)
	
	\bibitem{bombin2006}
	H. Bombin, M.A. Martin-Delgado,  {\em Homological error correction: Classical and quantum codes}, Journal of Mathematical Physics   {\bf 48} (2007), 052105; https://doi.org/10.1063/1.2731356
	
	\bibitem{tillich2009}
	J.P. Tillich and G. Z\'emor, {\em Quantum LDPC codes with positive rate and minimum distance proportional to $ n^{1/2} $}, In Information Theory, 2009. ISIT 2009. IEEE International Symposium, pages 799–803. IEEE, 2009.
	
	\bibitem{leslie2014}
	M. Leslie, {\em Hypermap-homology quantum codes}, International Journal of Quantum Information, Vol. 12, No. 01, 1430001 (2014).
	
	\bibitem{shor1995}
	P.W. Shor,  {\em Scheme for reducing decoherence in quantum memory}, Phys. Rev. A   {\bf 2} (1995), 2493-2496.
	
	\bibitem{hill}
	{R. Hill, {\em A first course in coding theory}, Oxford Applied Mathematics and Computing Science Series, Clarendon Press, 1997; 1st edition.}
	
	\bibitem{MichaelA.Nielsen2004}
	{Michael A. Nielsen and Isaac L. Chuang,} {\em {Quantum computation and quantum information}}, {Cambridge Series on Information and the Natural Sciences, Cambridge University Press, 2004; 1st edition.}
	
	\bibitem{Gottesman1997}
	{{Daniel Gottesman,}} {{{\em Stabilizer codes and quantum error correction}}, PhD thesis, California Institute of Technology, 1997.}
	
	\bibitem{Gruenbaum1987}
	{{\mbox{Branko Grünbaum and G. C. Shephard}}}, {{{\em Tilings and patterns}, Freeman and Co.,}} {{New York. 1987.}}
\end{thebibliography}
\end{document}